\documentclass{amsart}

\usepackage{graphicx}
\usepackage{amsmath}
\usepackage{amsfonts}
\usepackage{amssymb}
\usepackage{bbm}
\usepackage{epic}
\usepackage{amscd}

\usepackage{amsthm}

\usepackage[all]{xy}

\usepackage[margin=1.0in]{geometry}

\usepackage{appendix}

\usepackage{multirow}
\usepackage{tikz-cd}
\usepackage{url}

\newtheorem{theorem}{Theorem}[section]
\newtheorem{corollary}[theorem]{Corollary}
\newtheorem{lemma}[theorem]{Lemma}

\theoremstyle{definition}
\newtheorem{definition}[theorem]{Definition}
\newtheorem{remark}[theorem]{Remark}
\newtheorem{example}[theorem]{Example}

\DeclareMathOperator{\im}{im}

\DeclareMathOperator{\rank}{rank}

\DeclareMathOperator{\tr}{tr}

\DeclareMathOperator{\Spec}{Spec}

\DeclareMathOperator{\Homeo}{Homeo}
\DeclareMathOperator{\Arg}{Arg}

\newcommand{\Reals}{\mathbbm{R}}

\newcommand{\Ints}{\mathbbm{Z}}
\newcommand{\CP}{\mathbbm{CP}}

\newcommand{\Complex}{\mathbbm{C}}

\DeclareMathOperator{\Aut}{Aut}
\DeclareMathOperator{\Hom}{Hom}

\DeclareMathOperator{\Conf}{Conf}

\DeclareMathOperator{\Stab}{Stab}

\DeclareMathOperator{\Ad}{Ad}

\newcommand{\qx}{{\bf{i}}}
\newcommand{\qy}{{\bf{j}}}
\newcommand{\qz}{{\bf{k}}}

\DeclareMathOperator{\MCG}{MCG}

\DeclareMathOperator{\Out}{Out}

\begin{document}

\title{Holonomy perturbations of the Chern-Simons functional for lens
  spaces}

\author{David Boozer} 

\date{\today}

\maketitle

\begin{abstract}
We describe a scheme for constructing generating sets for Kronheimer
and Mrowka's singular instanton knot homology for the case of knots in
lens spaces.
The scheme involves Heegaard-splitting a lens space containing a knot
into two solid tori.
One solid torus contains a portion of the knot consisting of an
unknotted arc, as well as holonomy perturbations of the
Chern-Simons functional used to define the homology theory.
The other solid torus contains the remainder of the knot.
The Heegaard splitting yields a pair of Lagrangians in
the traceless $SU(2)$-character variety of the twice-punctured torus,
and the intersection points of these Lagrangians constitute the
generating set that we seek.
We illustrate the scheme by constructing generating sets for several
example knots.
Our scheme is a direct generalization of a scheme introduced by
Hedden, Herald, and Kirk for describing generating sets for knots in
$S^3$ in terms of Lagrangian intersections in the traceless
$SU(2)$-character variety for the 2-sphere with four punctures.
\end{abstract}

\setcounter{tocdepth}{2} 
\tableofcontents

\section{Introduction}

Singular instanton homology was introduced by Kronheimer
and Mrowka to describe knots in 3-manifolds
\cite{Kronheimer-1,Kronheimer-2,Kronheimer-3}.
Singular instanton homology is defined in terms of gauge theory,
but has important implications for Khovanov homology, a knot homology
theory that categorifies the Jones polynomial and that can be defined
in a purely combinatorial fashion.
Specifically, given a knot $K$ in $S^3$, Kronheimer and Mrowka showed
that there is a spectral sequence whose $E_2$ page is the reduced
Khovanov homology of the mirror knot $\overline{K}$, and that
converges to the singular instanton homology of $K$.
Using this spectral sequence, Kronheimer and Mrowka proved a key
property of Khovanov homology: a knot in $S^3$ is the unknot if and
only if its reduced Khovanov homology has rank 1.
This result is obtained by proving the analogous result for singular
instanton homology and then using the rank inequality implied by the
spectral sequence.

Calculations of singular instanton homology are generally difficult
to carry out, though some results are known.
For example, Kronheimer and Mrowka showed that the singular instanton
homology of an alternating knot in $S^3$ is isomorphic to the reduced
Khovanov homology of its mirror (modulo grading), since for such knots
the spectral sequence collapses at the $E_2$ page.
Hedden, Herald, and Kirk have described a scheme for
producing generating sets for singular instanton homology for a
variety of knots in $S^3$, which can sometimes be used to compute the
singular instanton homology itself \cite{Hedden-1}.
Their scheme works as follows.

Singular instanton homology is defined in terms of the Morse
complex of a perturbed Chern-Simons functional.
The unperturbed Chern-Simons functional is typically not Morse, so to
obtain a well-defined homology theory it is necessary to include a
small perturbation term.
For the case of knots in $S^3$, Hedden, Herald, and Kirk show how a
suitable perturbation can be constructed.
Their scheme involves Heegaard-splitting $S^3$ into a pair of solid
3-balls $B_1$ and $B_2$.
The ball $B_1$ contains a portion of the knot $K$ consisting of two
unknotted arcs, together with a specific holonomy perturbation of the
Chern-Simons functional.
The ball $B_2$ contains the remainder of the knot.
The Heegaard splitting of $S^3$ yields a pair of Lagrangians $L_1^\pi$ and
$L_2$ in the traceless $SU(2)$-character variety of the 2-sphere with
four punctures $R(S^2,4)$, a symplectic orbifold known as the
\emph{pillowcase} that is homeomorphic to the 2-sphere.
Specifically, the Lagrangians $L_1^\pi$ and $L_2$ describe conjugacy
classes of $SU(2)$-representations of the fundamental group of the
2-sphere with four punctures that extend to $B_1 - K$ and $B_2 - K$,
respectively.
In many cases, the points of intersection of $L_1^\pi$ and $L_2$
constitute a generating set for singular instanton homology.
The essential idea of the scheme is to confine all of the
perturbation data to a standard ball $B_1$ corresponding to a perturbed
Lagrangian $L_1^\pi$ that can be described explicitly.
The problem of constructing a generating set for a particular knot
then reduces to describing the Lagrangian $L_2$, a task that is
facilitated by the fact that the Chern-Simons functional is
unperturbed on the ball $B_2$.
In further work, Hedden, Herald, and Kirk define \emph{pillowcase
  homology} to be the Lagrangian Floer homology of the pair
$(L_1^\pi,L_2)$ \cite{Hedden-2}.
They conjecture that pillowcase homology is isomorphic to singular
instanton homology, and compute some examples that support this
conjecture.

In the present paper we generalize the scheme of Hedden, Herald, and
Kirk to the case of knots in lens spaces.
We Heegaard-split a lens space $Y$ containing a knot $K$ into two
solid tori $U_1$ and $U_2$.
The solid torus $U_1$ contains a portion of the knot consisting of
an unknotted arc, together with a specific holonomy perturbation.
The solid torus $U_2$ contains the remainder of the knot.
From the Heegaard splitting of $Y$ we obtain a pair of Lagrangians
$L_1^\pi$ and $L_2$ in the traceless $SU(2)$-character variety of the
twice-punctured torus $R(T^2,2)$, and in many cases the points of
intersection of $L_1^\pi$ and $L_2$ constitute a generating set for the
reduced singular instanton homology $I^\natural(Y,K)$.

To explain the details of our scheme, we must first define several
character varieties that are related to the Chern-Simons functional.
Critical points of the unperturbed Chern-Simons functional are flat
connections.
Gauge-equivalence classes of flat connections correspond to
conjugacy classes of homomorphisms
$\rho:\pi_1(Y - K \cup H \cup W) \rightarrow SU(2)$, where $H$ is a
small loop around $K$ and $W$ is an arc connecting $K$ to $H$, as
shown in Figure \ref{fig:KHW}, and the homomorphisms are required to
take loops around $K$ and $H$ to traceless matrices and loops around
$W$ to $-1$.
The space of such conjugacy classes forms a character variety that we
will denote by $R^\natural(Y,K)$.
We will refer to $H \cup W$ as an \emph{earring} that has been added to
the knot $K$.
The conditions on $\rho$ involving the earring are imposed in order
to avoid reducible connections; such connections prevent us from
obtaining a chain complex for singular instanton homology, with a
differential that squares to zero.
It will also be useful to define a character variety $R(Y,K)$ in which
we do not impose these conditions, and which consists of
conjugacy classes of homomorphisms $\rho:\pi_1(Y-K) \rightarrow SU(2)$
that take loops around $K$ to traceless matrices.

\begin{figure}
  \centering
  \includegraphics[scale=0.7]{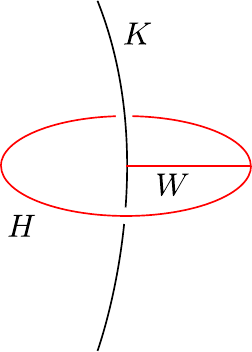}
  \caption{
    \label{fig:KHW}
    The knot $K$, loop $H$, and arc $W$.
  }
\end{figure}

The character variety $R^\natural(Y,K)$ is typically degenerate, in
which case the unperturbed Chern-Simons functional is not Morse.
We can render the Chern-Simons functional Morse by introducing
a suitable holonomy perturbation term that vanishes outside of
a small solid torus obtained by thickening a loop $P \subset Y$.
The net effect of the perturbation is to modify the corresponding
character variety: the critical points of the perturbed Chern-Simons
functional correspond to conjugacy classes of homomorphisms
$\rho:\pi_1(Y - K \cup H \cup W \cup P) \rightarrow SU(2)$, where
$\rho$ obeys the same conditions as for $R^\natural(Y,K)$ as well as
an additional condition involving the loop $P$ that we will describe
in Section \ref{sec:R-nat-pi}.
We will denote the character variety corresponding to the perturbed
Chern-Simons functional by $R_\pi^\natural(Y,K)$.

\begin{example}
For the trefoil $K$ in $S^3$, one can show that
\begin{align}
  \nonumber
  R(S^3,K) &= \{\textup{2 points}\}, &
  R^\natural(S^3,K) &= \{\textup{1 point}\} \amalg S^1, &
  R_\pi^\natural(S^3,K) &= \{\textup{3 points}\},
\end{align}
where the perturbation used to define $R_\pi^\natural(S^3,K)$ is as
described in Section \ref{sec:trefoil-s3}.
\end{example}

Our goal, then, is to devise an effective means of calculating
$R_\pi^\natural(Y,K)$.
We will view $(Y,K)$ as the result of gluing together two solid tori
$U_1 = S^1 \times D^2$ and $U_2 = S^1 \times D^2$.
The solid torus $U_1$ contains an unknotted arc $A_1$, the earring
$H \cup W$, and the holonomy perturbation loop $P$, as shown in
Figure \ref{fig:R-nat-pi}.
The solid torus $U_2$ contains a (possibly knotted) arc $A_2$.
We glue the two tori together via a homeomorphism
$\phi:(\partial U_1, \partial A_1) \rightarrow
(\partial U_2, \partial A_2)$ to obtain $(Y,K)$.

We define character varieties $R_\pi^\natural(U_1,A_1)$ and
$R(U_2,A_2)$ in analogy with $R_\pi^\natural(Y,K)$ and
$R(Y,K)$.
The character variety
$R_\pi^\natural(U_1,A_1)$ consists of conjugacy classes of homomorphisms
$\rho:\pi_1(U_1 - A_1 \cup H \cup W \cup P) \rightarrow SU(2)$ that
take loops around $A_1$ and $H$ to traceless matrices and loops around
$W$ to $-1$, and satisfy an additional requirement involving $P$ as
described in Section \ref{sec:R-nat-pi}.
The character variety $R(U_2,A_2)$ consists of conjugacy classes of
homomorphisms $\rho:\pi_1(U_2-A_2) \rightarrow SU(2)$ that
take loops around $A_2$ to traceless matrices.
We define a torus $T^2 := \partial U_1$ containing points
$\{p_1,p_2\} = \partial A_1$, and we define 
a corresponding character variety
$R(T^2,2)$ that consists of conjugacy classes of homomorphisms
$\rho:\pi_1(T^2 - \{p_1,p_2\}) \rightarrow SU(2)$ that take loops
around $p_1$ and $p_2$ to traceless matrices.

We define a map $R_\pi^\natural(U_1,A_1) \rightarrow R(T^2,2)$ by
pulling back along the inclusion
$(\partial U_1, \partial A_1) \hookrightarrow (U_1,A_1)$.
We define a map
$R(U_2,A_2) \rightarrow R(T^2,2)$
by pulling back along the composition of
$\phi:(\partial U_1, \partial A_1) \rightarrow
(\partial U_2, \partial A_2)$ with the inclusion
$(\partial U_2, \partial A_2) \hookrightarrow (U_2,A_2)$.
We similarly define maps
$R_\pi^\natural(Y,K) \rightarrow R_\pi^\natural(U_1,A_1)$ and
$R_\pi^\natural(Y,K) \rightarrow R(U_2,A_2)$ by pulling back along
inclusions.
We have a commutative diagram:
\begin{eqnarray}
  \label{diag:pullback}
  \begin{tikzcd}
    R_\pi^\natural(Y,K) \arrow{dr}{p} \arrow[bend left]{drr}
    \arrow[bend right]{ddr} & {} & {} \\
    {} & R_\pi^\natural(U_1,A_1) \times_{R(T^2, 2)} R(U_2,A_2) \arrow{d}
    \arrow{r} & R(U_2,A_2) \arrow{d} \\
    {} & R_\pi^\natural(U_1, A_1) \arrow{r} & R(T^2,2).
  \end{tikzcd}
\end{eqnarray}
Here $p$ is an induced map from $R_\pi^\natural(Y,K)$ to the fiber
product $R_\pi^\natural(U_1,A_1) \times_{R(T^2, 2)} R(U_2,A_2)$.
The character variety $R(T^2,2)$ is a symplectic orbifold that
generalizes the pillowcase, and the images of the maps
$R_\pi^\natural(U_1, A_1) \rightarrow R(T^2,2)$ and
$R(U_2,A_2) \rightarrow R(T^2,2)$ define Lagrangians $L_1^\pi$ and $L_2$
in $R(T^2,2)$.
We want to use diagram (\ref{diag:pullback}) to describe
$R_\pi^\natural(Y,K)$ in terms of the intersection points of these
Lagrangians.
Our first task is to obtain an explicit description of the character
variety $R(T^2,2)$.
We prove:

\begin{theorem}
\label{theorem:intro-1}
The character variety $R(T^2,2)$ contains a compact subset $P_3$ with
open, dense complement $P_4$.
The piece $P_3$ deformation retracts onto the pillowcase.
The piece $P_4$ is homeomorphic to $S^2 \times S^2 - \bar{\Delta}$,
where $\bar{\Delta} = \{(\hat{r},-\hat{r})\}$ is the antidiagonal.
(The pieces $P_3$ and $P_4$ are described explicitly in Theorems
\ref{theorem:space-P3} and \ref{theorem:space-P4}.)
\end{theorem}

\begin{remark}
The character variety $R(T^2,2)$ is
homeomorphic to the moduli space $M^{ss}(X,2)$ of semistable parabolic
bundles over an elliptic curve $X$, which is known to have the
structure of an algebraic variety isomorphic to
$\CP^1 \times \CP^1$ (see \cite{Boozer,Vargas}).
It follows that $R(T^2,2)$ is homeomorphic to $S^2 \times S^2$,
although this does not seem to be easy to show from our description of
this space.
We use the character variety $R(T^2,2)$, rather than the moduli space
$M^{ss}(X,2)$, since it is only for $R(T^2,2)$ that we can explicitly
describe the Lagrangians $L_1^\pi$ and $L_2$.
\end{remark}

Our next task is to explicitly describe the perturbed Lagrangian
$L_1^\pi$.
We prove:
\begin{theorem}
\label{theorem:intro-R-nat-pi}
The character variety $R_\pi^\natural(U_1, A_1)$ is homeomorphic to
$S^2$.
The map
$R_\pi^\natural(U_1, A_1) \rightarrow R(T^2,2)$
is an injective immersion away from the points of
$R_\pi^\natural(U_1, A_1)$ corresponding to the north and
south pole of $S^2$, which are mapped to the same point.
All representations in the image $L_1^\pi$ of the map are nonabelian.
(An explicit parameterization of $L_1^\pi$ is given
in Theorem \ref{theorem:L1}.)
\end{theorem}

Theorem \ref{theorem:intro-R-nat-pi} is closely analogous to a
corresponding result in the scheme of Hedden, Herald, and Kirk.
In their scheme, the perturbed Lagrangian analogous to our Lagrangian
$L_1^\pi$ is the image of $S^1$ under an injective immersion.
Their Lagrangian has single transverse double-point, and thus forms a
figure-8 in the pillowcase $R(S^2,4)$.
It is unclear whether it is possible, either in their scheme or ours,
to construct perturbed Lagrangians without double-points by choosing
different holonomy perturbations.
Theorem \ref{theorem:intro-R-nat-pi} also yields:

\begin{corollary}
\label{cor:glue}
The map $p: R_\pi^\natural(Y,K) \rightarrow
R_\pi^\natural(U_1,A_1) \times_{R(T^2, 2)} R(U_2,A_2)$
in diagram (\ref{diag:pullback}) is injective.
\end{corollary}

\begin{proof}
Consider a point $([\rho_1],[\rho_2])$ in
$R_\pi^\natural(U_1,A_1) \times_{R(T^2, 2)} R(U_2,A_2)$, so $\rho_1$
and $\rho_2$ pull back to the same homomorphism
$\rho_{12}:\pi_1(T^2-\{p_1,p_2\}) \rightarrow SU(2)$.
One can show (see \cite{Hedden-1} Lemma 4.2) that the fiber
$p^{-1}([\rho_1],[\rho_2])$ is homeomorphic to the double coset space
$\Stab(\rho_1)\backslash\Stab(\rho_{12})/\Stab(\rho_2)$, where
\begin{align}
  \nonumber
  \Stab(\rho) = \{g \in SU(2) \mid
  \textup{$g\rho(x)g^{-1} = \rho(x)$ for all $x$ in the domain of
    $\rho$}\}.
\end{align}
The center of $SU(2)$ is $Z(SU(2)) = \{\pm 1\}$.
By Theorem \ref{theorem:intro-R-nat-pi} we have that
$\Stab(\rho_{12}) = Z(SU(2))$, and
$Z(SU(2)) \subseteq \Stab(\rho_i) \subseteq \Stab(\rho_{12})$, so
$\Stab(\rho_i) = \Stab(\rho_{12}) = Z(SU(2))$ and thus the fibers of
$p$ are points.
\end{proof}

By introducing a suitable holonomy perturbation, we obtain a finite
character variety $R_\pi^\natural(Y,K)$, each point of which
corresponds to a gauge-orbit of connections that are critical points
of the perturbed Chern-Simons functional.
In order for $R_\pi^\natural(Y,K)$ to serve as a generating set for
singular instanton homology, each point in $R_\pi^\natural(Y,K)$ must
be nondegenerate; that is, at each connection representing a point in
$R_\pi^\natural(Y,K)$ we want the Hessian of the perturbed
Chern-Simons functional to be nondegenerate when restricted to a
complement of the tangent space to the gauge-orbit of that connection.
We show that there is a simple criterion for determining when
a point $R_\pi^\natural(Y,K)$ is nondegenerate.
Recall that we defined the Lagrangian $L_2$ to be the image of
$R(U_2,A_2) \rightarrow R(T^2,2)$.
If $R(U_2,A_2) \rightarrow R(T^2,2)$ is injective and
$[\rho] \in L_1^\pi \cap L_2 \subset R(T^2,2)$ is not the double-point of
$L_1^\pi$, then by Corollary \ref{cor:glue} the point $[\rho]$ is the
image of a unique point
$[\tilde{\rho}] \in R_\pi^\natural(Y,K)$ under the pullback map
$R_\pi^\natural(Y,K) \rightarrow R(T^2,2)$.
We prove:

\begin{theorem}
\label{theorem:intro-transverse}
Suppose $R(U_2,A_2) \rightarrow R(T^2,2)$ is an injective
immersion and $[\rho] \in L_1^\pi \cap L_2$
is the image of a regular point of $R(U_2,A_2)$ and is not the
double-point of $L_1^\pi$.
Then the unique preimage $[\tilde{\rho}]$ of $[\rho]$ under the
pullback map $R_\pi^\natural(Y,K) \rightarrow R(T^2,2)$ is
nondegenerate if and only if the intersection of $L_1^\pi$ with $L_2$ at
$[\rho]$ is transverse.
\end{theorem}

Collecting these results, we find if the hypotheses of
Theorem \ref{theorem:intro-transverse} are satisfied for every
point in $L_1^\pi \cap L_2$, then every point in
$R_\pi^\natural(Y,K)$ is nondegenerate and the pullback map
$R_\pi^\natural(Y,K) \rightarrow R(T^2,2)$ is injective with image
$L_1^\pi \cap L_2$.
Thus we obtain:

\begin{corollary}
If the hypotheses of
Theorem \ref{theorem:intro-transverse} are satisfied for every
point in $L_1^\pi \cap L_2$, then $R_\pi^\natural(Y,K)$ is a generating
set for $I^\natural(Y,K)$ consisting of $|L_1^\pi \cap L_2|$
generators.
\end{corollary}

Our scheme is particularly well-suited for the case of $(1,1)$-knots.
By definition, a $(1,1)$-knot is a knot $K$ in a lens space $Y$ that
has a Heegaard splitting into a pair of solid tori $U_1,U_2 \subset Y$
such that the components $U_1 \cap K$ and $U_2 \cap K$ of the knot in
each solid torus are unknotted arcs.
It is known that $(1,1)$-knots include all torus knots and 2-bridge
knots.

We can construct $(1,1)$-knots by taking $(U_2,A_2)$ to be a
copy of $(U_1,A_1)$ without the earring $H \cup W$ or the
perturbation loop $P$, and we can explicitly describe the
corresponding Lagrangian $L_2$ as follows.
We first define a character variety $R(U_1,A_1)$ that consists of
conjugacy classes of homomorphisms $\pi_1(U_1-A_1) \rightarrow SU(2)$
that take loops around $A_1$ to traceless matrices.
We define a map $R(U_1,A_1) \rightarrow R(T^2,2)$ by pulling back
along the inclusion
$(\partial U_1, \partial A_1) \hookrightarrow (U_1,A_1)$.
The image of this map defines a Lagrangian $L_1$ in $R(T^2,2)$.
We can view $R(U_1,A_1)$ and $L_1$ as ``unperturbed'' versions of
$R_\pi^\natural(U_1,A_1)$ and $L_1^\pi$.
Since
$(T^2,\{p_1,p_2\}) := (\partial U_1, \partial A_1)$ and
there is a natural identification
$(\partial U_2, \partial A_2) \xrightarrow{\sim} (T^2,\{p_1,p_2\})$,
the gluing map
$\phi:(\partial U_1, \partial A_1) \rightarrow
(\partial U_2, \partial A_2)$
defines an element $[\phi]$ of the mapping class group
$\MCG_2(T^2)$ of the twice-punctured torus.
The group $\MCG_2(T^2)$ acts on $R(T^2,2)$ from the right in
a way that we explicitly describe in Section \ref{sec:mcg}, and
the Lagrangian $L_2$, which we defined to be the image of
the map $R(U_2,A_2) \rightarrow R(T^2,2)$, is given by
$L_2 = L_1 \cdot [\phi]$.
We prove results that explicitly describe the character variety
$R(U_1,A_1)$ and the Lagrangian $L_1$:

\begin{theorem}
\label{theorem:intro-R}
The character variety $R(U_1, A_1)$ is homeomorphic to the closed disk
$D^2$.
The map $R(U_1,A_1) \rightarrow R(T^2,2)$ is injective and is an
immersion on the interior of $R(U_1,A_1)$.
(An explicit parameterization of the image $L_1$ of the map is given
in Theorem \ref{theorem:Ld}.)
\end{theorem}

\begin{theorem}
\label{theorem:intro-reg-R}
The character variety $R(U_1, A_1)$ is regular on its interior.
\end{theorem}

From Theorems \ref{theorem:intro-R} and \ref{theorem:intro-reg-R}, we
obtain a Corollary to Theorem \ref{theorem:intro-transverse} for the
special case of $(1,1)$-knots.
Recall that for a $(1,1)$ knot we can describe the Lagrangian
$L_2$ as the image $L_2 = L_1 \cdot [\phi]$ of $L_1$ under the action
of the mapping class group element $[\phi] \in \MCG_2(T^2)$, where
$\phi:(\partial U_1, \partial A_1) \rightarrow
(\partial U_2, \partial A_2)$ is the gluing map.
We then have:

\begin{corollary}
\label{cor:1-1-knot}
For a $(1,1)$-knot $K$, if $L_1^\pi$ intersects
$L_2 = L_1 \cdot [\phi]$
transversely away from the double-point of $L_1^\pi$, then
$R_\pi^\natural(Y,K)$ is a generating set for $I^\natural(Y,K)$
consisting of $|L_1^\pi \cap L_2|$ generators.
\end{corollary}

Since we have explicit descriptions of the character variety
$R(T^2,2)$, the Lagrangians $L_1^\pi$ and
$L_1$, and the action of the mapping class group $\MCG_2(T^2)$ on
$R(T^2,2)$, Corollary \ref{cor:1-1-knot} provides us with a practical
scheme for calculating generating sets for $I^\natural(Y,K)$ for any
$(1,1)$-knot $K$ in any lens space $Y$.
In Section \ref{sec:examples} we illustrate this scheme by
calculating generating sets for several example $(1,1)$-knots.
We first rederive known results for knots in $S^3$: we produce
generating sets for the unknot (one generator) and trefoil (three
generators), which allow us to compute the singular instanton homology
for these knots.
Next we consider knots in lens spaces $L(p,1)$.
We prove:

\begin{theorem}
\label{theorem:intro-unknot}
If $p$ is not a multiple of $4$, then the rank of
$I^\natural(L(p,1),U)$ for the unknot $U$ in the lens space $L(p,1)$
is at most $p$.
\end{theorem}

A knot $K$ in a lens space $L(p,q)$ is said to be \emph{simple} if the
lens space can be Heegaard-split into solid tori $U_1$ and $U_2$
with meridian disks $D_1$ and $D_2$ such that $D_1$ intersects $D_2$
in $p$ points and $K \cap U_i$ is an unknotted arc in disk $D_i$ for
$i=1,2$ (see \cite{Hedden-0}).
Up to isotopy, there is exactly one simple knot in each nonzero
homology class of $H_1(L(p,q);\Ints) = \Ints_p$.
We prove:

\begin{theorem}
\label{theorem:intro-simple}
If $K$ is the unique simple knot representing the homology class
$1 \in \Ints_p = H_1(L(p,1);\Ints)$ of the lens space $L(p,1)$, then
the rank of $I^\natural(L(p,1),K)$ is at most $p$.
\end{theorem}

To our knowledge, Theorems \ref{theorem:intro-unknot} and
\ref{theorem:intro-simple} give the first rank bounds on instanton
homology for knots in 3-manifolds other than $S^3$.
For a simple knot $K$ in the lens space $Y = L(p,q)$, the knot Floer
homology $\widehat{HFK}(Y,K)$ has rank $p$ (see \cite{Hedden-0}).
Thus, Theorem \ref{theorem:intro-simple} is consistent with Kronheimer
and Mrowka's conjecture that for a knot $K$ in a 3-manifold $Y$, the
ranks of $I^\natural(Y,K)$ and
$\widehat{HFK}(Y,K)$ are the same (see \cite{Kronheimer-0} Section
7.9).
There is a combinatorial method for computing the Floer
homology for arbitrary $(1,1)$-knots in $S^3$ (see \cite{Goda}), and
it would be interesting to test Kronheimer and Mrowka's conjecture by
using our results to derive rank bounds for the singular instanton
homology of such knots.

As a more ambitious application of our results, we hope to generalize
Khovanov homology to links in lens spaces.
In recent work \cite{Hedden-3}, Hedden, Herald, Hogancamp, and Kirk
show that Bar Natan's functor from the tangle cobordism category for
2-tangles in the 3-ball to chain complexes can be factored through the
twisted Fukaya category for the pillowcase $R(S^2,4)$.
We conjecture that an analogous factorization result holds for the
twisted Fukaya category of $R(T^2,2)$; if so, the corresponding
functor would provide a natural way of generalizing Khovanov homology
to links in lens spaces.
Using the results presented here, we can already describe some
of the relevant structure of the Fukaya category of $R(T^2,2)$.
The information we have obtained regarding this Fukaya
category suggests that the factorization result of \cite{Hedden-3}
does indeed generalize, but we have not proven this.
Based on clues provided by this information, we have shown that the
reduced Khovanov homology of a link in $S^3$ can be recovered from a
chain complex constructed from a description of the link in terms of a
1-tangle in the annulus \cite{Boozer-2}.
We hope that further investigation of the Fukaya category of $R(T^2,2)$
will enable us to construct a homology theory for links in lens spaces
that generalizes Khovanov homology.

\section{The group $SU(2)$}

Here we briefly review some basic facts about the group $SU(2)$.
We define $SU(2)$-matrices $\qx$, $\qy$, and $\qz$ by
\begin{align}
  \nonumber  
  \qx &= -i\sigma_x, &
  \qy &= -i\sigma_y, &
  \qz &= -i\sigma_z,
\end{align}
where
$\sigma_x$, $\sigma_y$, and $\sigma_z$ are the Pauli spin matrices:
\begin{align}
  \nonumber  
  \sigma_x &=
  \left(\begin{array}{cc}
    0 & 1 \\
    1 & 0
  \end{array}\right), &
  \sigma_y &=
  \left(\begin{array}{cc}
    0 & -i \\
    i & 0
  \end{array}\right), &
  \sigma_z &=
  \left(\begin{array}{cc}
    1 & 0 \\
    0 & -1
  \end{array}\right).
\end{align}
The matrices $\qx$, $\qy$, and $\qz$ satisfy the quaternion
multiplication laws $\qx^2 = \qy^2 = \qz^2 = \qx \qy \qz = -1$.
Any $SU(2)$-matrix $A$ can be uniquely expressed as
\begin{align}
  \nonumber  
  A = t + x\,\qx + y\,\qy + z\,\qz,
\end{align}
where
$(t,x,y,z) \in S^3 =
\{(t,x,y,z) \in \Reals^4 \mid t^2 + x^2 + y^2 + z^2 = 1\}$,
and thus we may identify $SU(2)$ with the space of unit quaternions.
We will refer to $t$ and $x\,\qx + y\,\qy + z\,\qz$ as the
\emph{scalar} and \emph{vector}
parts of the matrix $A$, respectively.
Note that $\tr(A) = 2t$, so traceless $SU(2)$-matrices are precisely
those for which the scalar part is zero.
It follows that traceless $SU(2)$-matrices are parameterized by unit
vectors in $\Reals^3$, and we will frequently pass back and forth
between traceless matrices
$a = a_x\,\qx + a_y\,\qy + a_z\,\qz \in SU(2)$ and their corresponding
unit vectors $\hat{a} = (a_x,a_y,a_z) \in S^2$.

We can define a surjective group homomorphism
$SU(2) \rightarrow SO(3)$ by $g \mapsto (\hat{v} \mapsto \hat{v}')$,
where the unit vectors $\hat{v} = (v_x, v_y, v_z)$ and
$\hat{v}' = (v_x', v_y', v_z')$ are related by
\begin{align}
  \nonumber  
  g(v_x\,\qx + v_y\,\qy + v_z\,\qz)g^{-1} =
  v_x'\,\qx + v_y'\,\qy + v_z'\,\qz.
\end{align}
In general, conjugating an arbitrary $SU(2)$-matrix preserves the
scalar part of the matrix and rotates the vector part of the matrix:
\begin{align}
  \nonumber  
  g(t + r_x\,\qx + r_y\,\qy + r_z\,\qz)g^{-1} =
  t + r_x'\,\qx + r_y'\,\qy + r_z'\,\qz,
\end{align}
where $(r_x',r_y',r_z')$ is given by multiplying
$(r_x,r_y,r_z)$ by the $SO(3)$-matrix corresponding to $g \in SU(2)$.
We will thus sometimes describe conjugation in terms of the
corresponding rotation performed on the vector part of an
$SU(2)$-matrix.
We use the conjugation action to prove the following result:

\begin{lemma}
\label{lemma:product-surj}
Any $SU(2)$-matrix can be expressed as a product of two traceless
$SU(2)$-matrices.
\end{lemma}

\begin{proof}
Define a product map $SU(2) \times SU(2) \rightarrow SU(2)$,
$(a,b) \mapsto ab$.
If we restrict the product map to traceless $SU(2)$-matrices,
represent each traceless $SU(2)$-matrix as a point in $S^2$, and
represent their product as a point in $S^3$, we obtain a map
$S^2 \times S^2 \rightarrow S^3$ given by
\begin{align*}
  (\hat{a},\,\hat{b}) \mapsto
  (-\hat{a}\cdot \hat{b},\, \hat{a} \times \hat{b}).
\end{align*}
This map is surjective, since
\begin{align*}
  (\hat{x},\,\hat{x}\cos\theta + \hat{y}\sin\theta) \mapsto
  (-\cos\theta,\, \hat{z}\sin\theta)
\end{align*}
and the product map is equivariant under conjugation.
\end{proof}

\section{Character varieties}
\label{sec:character-varieties}

\subsection{The character variety $R(T^2,2)$}
\label{sec:R2}

Our first task is to understand the structure of $R(T^2,2)$, the
traceless $SU(2)$-character variety of the twice-punctured torus.
In general, we make the following definition:

\begin{definition}
Given a surface $S$ with $n$ distinct marked points
$p_1, \cdots, p_n \in S$, we define the character variety $R(S,n)$ to
be the space of conjugacy classes of homomorphisms
$\rho:\pi_1(S - \{p_1, \cdots, p_n\}) \rightarrow SU(2)$ that take
loops around the marked points to traceless $SU(2)$-matrices.
\end{definition}

Before examining the space $R(T^2,2)$, we first consider the simpler
space $R(T^2) := R(T^2,0)$, which is known as the \emph{pillowcase}.
We have the following well-known result:

\begin{theorem}
\label{theorem:pillowcase}
The pillowcase $R(T^2)$ is homeomorphic to $S^2$.
\end{theorem}

\begin{proof}
The fundamental group of $T^2$ is
$\pi_1(T^2) = \langle A, B \mid ABA^{-1}B^{-1} = 1\rangle$,
where $A$ and $B$ are represented by the two fundamental cycles.
A homomorphism $\rho:\pi_1(T^2) \rightarrow SU(2)$ is uniquely
determined by the pair of matrices $(\rho(A),\rho(B))$, which for
simplicity we will also denote by $(A,B)$.
Since $A$ and $B$ commute, any conjugacy class $[\rho] \in R(T^2)$ has
a representative $(A,B)$ of the form
\begin{align}
  \nonumber
  A &= \cos \alpha + \sin\alpha\,\qz, &
  B &= \cos \beta + \sin\beta\,\qz
\end{align}
for some angles $\alpha$ and $\beta$.
These equations are invariant under the
replacements $\alpha \rightarrow \alpha + 2\pi$ and
$\beta \rightarrow \beta + 2\pi$, and we can simultaneously flip the
signs of $\alpha$ and $\beta$ by conjugating by $\qx$, so we
obtain the following identifications:
\begin{align}
  \nonumber
  (\alpha,\beta) \sim (\alpha+2\pi,\beta), &&
  (\alpha,\beta) \sim (\alpha,\beta+2\pi), &&
  (\alpha,\beta) \sim (-\alpha,-\beta).
\end{align}
We can thus restrict to a fundamental domain in which
$(\alpha,\beta) \in [0,2\pi] \times [0,\pi]$, with edges identified as
shown in Figure \ref{fig:pillowcase}.
From Figure \ref{fig:pillowcase} it is clear that this space is
homeomorphic to $S^2$.
\end{proof}

\begin{figure}[t]
  \centering
  \includegraphics[scale=0.7]{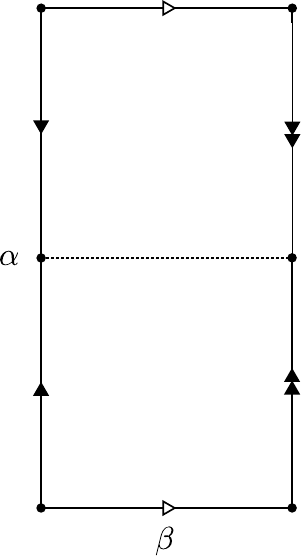}
  \caption{
    \label{fig:pillowcase}
    The pillowcase $R(T^2)$.
    The black dots indicate the four pillowcase points.
  }
\end{figure}

We will refer to the four points
$[A,B] = [\pm 1, \pm 1] \in R(T^2)$ as \emph{pillowcase points}.
One can show that the character variety $R(S^2,4)$ that is used in the
work of Hedden, Herald, and Kirk is also described
by a rectangle with edges identified as shown in
Figure \ref{fig:pillowcase} (see, for example, \cite{Hedden-1} Section
3.1), so both $R(T^2)$ and $R(S^2,4)$ are referred to as the
\emph{pillowcase}.

We now consider the space $R(T^2,2)$.
The fundamental group of the twice-punctured torus is
\begin{align}
  \nonumber
  \pi_1(T^2 - \{p_1, p_2\}) =
  \langle A, B, a, b \mid ABA^{-1}B^{-1}ab = 1 \rangle,
\end{align}
where $p_1$ and $p_2$ denote the puncture points, $A$ and $B$ denote
the fundamental cycles of the torus, and $a$ and $b$ denote loops
around the punctures $p_1$ and $p_2$, as shown in Figure
\ref{fig:pi1-t2}.
As above, we will use the same notation for generators of the
fundamental group and their images under $\rho$; for example, we
denote $\rho(A)$ by $A$.
A homomorphism
$\rho:\pi_1(T^2 - \{p_1, p_2\}) \rightarrow SU(2)$ is thus specified
by $SU(2)$-matrices $(A,B,a,b)$ such that $a$ and $b$ are traceless
and $ABA^{-1}B^{-1}ab = 1$,
and we will sometimes denote a homomorphism $\rho$ by the
corresponding list of matrices $(A,B,a,b)$.

\begin{figure}[t]
  \centering
  \includegraphics[scale=0.7]{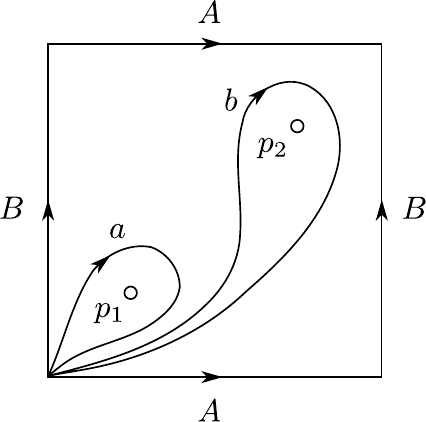}
  \caption{
    \label{fig:pi1-t2}
    Cycles corresponding to the generators $A$, $B$, $a$, $b$ of the
    fundamental group $\pi_1(T^2 - \{p_1,p_2\})$.
  }
\end{figure}

The structure of $R(T^2,2)$ can be understood by considering
the fibers of the map
$\mu:R(T^2,2) \rightarrow [-1,1]$ defined by
\begin{align*}
  \mu([A,B,a,b]) = (1/2) \tr (ABA^{-1}B^{-1}) = (1/2)\tr ((ab)^{-1}).
\end{align*}
In particular, it is convenient to decompose $R(T^2,2)$ into
the disjoint union of an open piece $P_4 = \mu^{-1}([-1,1))$ and a
closed piece $P_3 = \mu^{-1}(1)$.
The notation for these pieces is motivated by the fact that, as we
will see, the piece $P_4$ is four-dimensional and the piece $P_3$ is
three-dimensional.
We will describe the topology of the pieces $P_3$ and $P_4$ and
define coordinate systems on each piece that are useful for performing
calculations.

\subsubsection{The piece $P_4 \subset R(T^2,2)$}
\label{sssec:P4}

We define the piece $P_4 \subset R(T^2,2)$ to be the set of conjugacy
classes $[\rho] \in R(T^2,2)$ such that $\mu([\rho]) \in [-1,1)$.
For any representative $(A,B,a,b)$ of a given conjugacy class
$[\rho] \in P_4$, the matrices $A$ and $B$ do not commute.
This fact can be used to choose a canonical representative of each
conjugacy class in $P_4$:

\begin{lemma}
Any conjugacy class $[\rho] \in P_4$ has a unique representative
$(A,B,a,b)$ for which
\begin{align}
  \label{eqn:AB}
  A &=
  r\cos\alpha + \sqrt{1 - r^2}\,\qx + r\sin\alpha\,\qz, &
  B &= \cos\beta + \sin\beta\,\qz,
\end{align}
where $\alpha \in [0,2\pi)$, $\beta \in (0,\pi)$, and $r \in [0,1)$.
\end{lemma}

\begin{proof}
Since $[\rho] \in P_4$, for any representative of $[\rho]$
the matrices $A$ and $B$ do not commute.
Given an arbitrary representative, first
conjugate so that the coefficients of $\qx$ and $\qy$ in
$B$ are zero and the coefficient of $\qz$ is positive, and
then rotate about the $z$-axis so the coefficient of $\qy$ in
$A$ is zero and the coefficient of $\qx$ in $A$ is positive.
The restrictions on the ranges of $\beta$ and $r$ follow from the fact
that the matrices $A$ and $B$ do not commute.
The uniqueness of the representative follows from the fact that
the coefficients of $\qx$ in $A$ and $\qz$ in $B$ are nonzero.
\end{proof}

We can use the canonical representatives of conjugacy classes in $P_4$
to define maps
$q_1:P_4 \rightarrow SU(2)\times SU(2)$ and
$q_2:P_4 \rightarrow S^2 \times S^2$:
\begin{align}
  \nonumber
  q_1([\rho]) &= (A,B), &
  q_2([\rho]) &= (\hat{a}, \hat{b}),
\end{align}
where
$(A,B,a,b)$ is the canonical representative of
$[\rho]$, and $\hat{a} = (a_x, a_y, a_z)$ and
$\hat{b} = (b_x, b_y, b_z)$ are the unit vectors corresponding to the
traceless matrices $a$ and $b$:
\begin{align}
  \nonumber
  a &= a_x\,\qx + a_y\,\qy + a_z\,\qz, &
  b &= b_x\,\qx + b_y\,\qy + b_z\,\qz.
\end{align}
Note that we cannot extend the maps $q_1$ and $q_2$ to all of
$R(T^2,2)$, since our choice of canonical representative relies on the
fact that the matrices $A$ and $B$ do not commute.

To describe the piece $P_4$, we will
show that the map $q_2:P_4 \rightarrow S^2 \times S^2$ is
injective and identify its image.
This requires two Lemmas that describe the image of
$q_1:P_4 \rightarrow SU(2) \times SU(2)$ on the
fibers of $\mu:P_4 \rightarrow [-1,1)$:

\begin{lemma}
\label{lemma:p1-iso-pt}
The space $q_1(\mu^{-1}(-1))$ consists of the single point
$(\qx,\,\qz)$.
\end{lemma}

\begin{proof}
Consider a point $[\rho] \in \mu^{-1}(-1)$.
From equation (\ref{eqn:AB}) for the canonical representative
$(A,B,a,b)$ of $[\rho]$, we find that
\begin{align}
  \nonumber
  \mu([\rho]) = -1 = (1/2)\tr(ABA^{-1}B^{-1}) =
  \cos 2\beta + r^2(1- \cos 2\beta).
\end{align}
Thus $r=0$ and $\beta = \pi/2$.
Substituting these values into equation (\ref{eqn:AB}), we obtain the
desired result.
\end{proof}

\begin{lemma}
\label{lemma:p1-iso-s2}
For $t \in (-1,1)$ we can define a map
$q_1(\mu^{-1}(t)) \rightarrow S^2$,
$(A,B) \mapsto \hat{v}$, where the unit vector
$\hat{v} = (v_x, v_y, v_z)$ is the
direction of the vector part of $ABA^{-1}B^{-1}$:
\begin{align}
  \nonumber
  ABA^{-1}B^{-1} =
  t + \sqrt{1-t^2}\,(v_x\,\qx + v_y\,\qy + v_z\,\qz).
\end{align}
This map is a homeomorphism.
\end{lemma}

\begin{proof}
Consider a point $[\rho] \in \mu^{-1}(t)$ for $t \in (-1,1)$.
From equation (\ref{eqn:AB}) for the canonical representative
$(A,B,a,b)$ of $[\rho]$, we find that
\begin{align}
  \label{eqn:t-r}
  \mu([\rho]) = t = (1/2)\tr(ABA^{-1}B^{-1}) =
  \cos 2\beta + r^2(1- \cos 2\beta).
\end{align}
We solve equation (\ref{eqn:t-r}) for $r$ to obtain
\begin{align}
  \label{eqn:r}
  r = \left(\frac{t - \cos2\beta}{1 - \cos 2\beta}\right)^{1/2}.
\end{align}
From equation (\ref{eqn:r}), we see that for a fixed value of
$t \in (-1,1)$ the parameter $\beta$ must lie in the range
$[\beta_0,\pi - \beta_0]$, where we have defined
\begin{align}
  \label{eqn:phi0}
  \beta_0 := (1/2)\cos^{-1} t \in (0,\pi/2).
\end{align}
Using equations (\ref{eqn:AB}), (\ref{eqn:r}), and (\ref{eqn:phi0}),
we find that the matrices $A$ and $B$ can be expressed as
\begin{align}
  \label{eqn:AB-t}
  A &=
  \left(\frac{\cos2\beta_0 - \cos2\beta}{1 - \cos2\beta}\right)^{1/2}
  (\cos\alpha + \sin\alpha\,\qz) +
   \left(\frac{1-\cos2\beta_0}{1 - \cos2\beta}\right)^{1/2}\qx, &
  B &= \cos\beta + \sin\beta\,\qz,
\end{align}
where $(\alpha,\beta) \in [0,2\pi)\times [\beta_0,\pi-\beta_0]$.
Define a space
\begin{align}
  \nonumber
  X = \{(\alpha, \beta) \in
  [0,2\pi] \times
  [\beta_0,\pi-\beta_0]
  \}/{\sim},
\end{align}
where the equivalence relation $\sim$ is defined such that the bottom
edge of the rectangle $[0,2\pi] \times [\beta_0,\pi-\beta_0]$ is
collapsed to a point, the top edge is collapsed to a point, and the
left and right edges are identified:
\begin{align}
  \nonumber
  (\alpha,\beta_0) \sim (0,\beta_0), &&
  (\alpha,\pi-\beta_0) \sim (0,\pi-\beta_0), &&
  (0,\beta) \sim (2\pi,\beta).
\end{align}
We note that $\mu^{-1}(t)$ is nonempty for all $t \in (-1,1)$, since
Lemma \ref{lemma:product-surj} shows that we can always find traceless
$SU(2)$-matrices $a$ and $b$ such that $ABA^{-1}B^{-1}ab = 1$.
Define a map
$X \rightarrow q_1(\mu^{-1}(t))$,
$[\alpha,\beta] \mapsto (A,B)$, where
$A$ and $B$ are given by equation (\ref{eqn:AB-t}).
From equation (\ref{eqn:AB-t}), it is clear that this map is
well-defined and is a homeomorphism.

Using equations (\ref{eqn:AB}) and (\ref{eqn:r}), a calculation shows
that
\begin{align}
  \nonumber
  ABA^{-1}B^{-1} =
  t + \sqrt{1-t^2}\,(v_x\,\qx + v_y\,\qy + v_z\,\qz),
\end{align}
where the unit vector $\hat{v} = (v_x, v_y, v_z) \in S^2$ is given by
\begin{align}
  \label{eqn:vhat}
  \hat{v} =
  (\sqrt{1-z(\beta)^2}\,\sin(\alpha+\beta),\,
  -\sqrt{1-z(\beta)^2}\,\cos(\alpha+\beta),\,
  z(\beta))
\end{align}
and we have defined a diffeomorphism
 $z:[\beta_0,\pi-\beta_0] \rightarrow [-1,1]$ by
\begin{align}
  \nonumber
  z(\beta) = -\sqrt{\frac{1-t}{1+t}} \cot\beta.
\end{align}
Define a map
$X \rightarrow S^2$,
$[\alpha,\beta] \mapsto \hat{v}$,
where $\hat{v}$ is given by equation (\ref{eqn:vhat}).
From equation (\ref{eqn:vhat}), it is clear that this map is
well-defined and is a homeomorphism.
Composing the inverse of the map $X \rightarrow q_1(\mu^{-1}(t))$ with
the map $X \rightarrow S^2$, we obtain the desired result.
\end{proof}

We can now describe the topology of the piece $P_4$:

\begin{theorem}
\label{theorem:space-P4}
The space $P_4$ is homeomorphic to
$S^2 \times S^2 - \bar{\Delta}$, where
$\bar{\Delta} = \{(\hat{r},-\hat{r}) \}$ is the antidiagonal.
All representations in $P_4$ are nonabelian.
\end{theorem}

\begin{proof}
Consider the map
$q_2:P_4 \rightarrow S^2 \times S^2$.
Clearly the image of $q_2$ lies in
$S^2 \times S^2 - \bar{\Delta}$, since if
$q_2([\rho]) \in \bar{\Delta}$ then $b = a^{-1}$, which implies that
$\mu([\rho]) = (1/2)\tr((ab)^{-1}) = 1$ and hence
$[\rho] \notin P_4$.
We can define an inverse map
$S^2 \times S^2 - \bar{\Delta} \rightarrow P_4$ as follows.
Given a point $(\hat{a}, \hat{b}) \in S^2 \times S^2 - \bar{\Delta}$,
define traceless matrices
$a = a_x\,\qx + a_y\,\qy + a_z\,\qz$ and
$b = b_x\,\qx + b_y\,\qy + b_z\,\qz$ corresponding to
$\hat{a} = (a_x, a_y, a_z)$ and $\hat{b} = (b_x, b_y, b_z)$.
Then
\begin{align}
  \nonumber
  ab =
  t - v_x\,\qx - v_y\,\qy - v_z\,\qz,
\end{align}
where
$t := -\hat{a}\cdot\hat{b}$ and
$\vec{v} = (v_x, v_y, v_z) := -\hat{a} \times \hat{b}$.
If $t=-1$ then map $(\hat{a},\hat{b})$ to
$[\qx,\qz, a, b]$,
otherwise map $(\hat{a},\hat{b})$ to
$[A,B,a,b]$, where $A$ and $B$ are determined from $t$ and
$\hat{v} := \vec{v}/|\vec{v}| \in S^2$ via the homeomorphism
$q_1(\mu^{-1}(t)) \rightarrow S^2$ defined in Lemma
\ref{lemma:p1-iso-s2}.
By Lemmas \ref{lemma:p1-iso-pt} and \ref{lemma:p1-iso-s2}, this
inverse map is well-defined.
The fact that all representations in $P_4$ are nonabelian
is clear from the definition of the space $P_4$.
\end{proof}

Our main application of Theorem \ref{theorem:space-P4} will be to use
$(\hat{a},\hat{b})$ as coordinates on the piece $P_4$.

\subsubsection{The piece $P_3 \subset R(T^2,2)$}
\label{sssec:P3}

We define the piece $P_3 \subset R(T^2,2)$ to be the set of conjugacy
classes $[\rho] \in R(T^2,2)$ such that $\mu([\rho]) = 1$.
For any representative $(A,B,a,b)$ of a given conjugacy class
$[\rho] \in P_3$, the matrices $A$ and $B$ commute.
We can therefore define a map $q:P_3 \rightarrow R(T^2)$ by
\begin{align*}
  q([A,B,a,b]) = [A,B].
\end{align*}
We will describe the topology of the piece $P_3$ by considering the
fibers of the map $q$.
In particular, we will show that $P_3$ is homeomorphic to the
following space:

\begin{definition}
\label{def:space-Y}
We define a space $Y$ by
\begin{align}
  \nonumber
  Y = \{(\alpha,\beta,z) \mid
  \alpha \in [0,2\pi],\,
  \beta \in [0,\pi],\,
  |z| \leq \sin^2\alpha + \sin^2\beta\}/{\sim},
\end{align}
where the equivalence relation $\sim$ is defined such that
\begin{align}
  \nonumber
  (\alpha,0,z) \sim (2\pi-\alpha,0,-z), &&
  (\alpha,\pi,z) \sim (2\pi-\alpha,\pi,-z), &&
  (0,\beta,z) \sim (2\pi,\beta,z).
\end{align}
\end{definition}

The space $Y$ is depicted in Figure \ref{fig:space-Y}.
We can think of $Y$ as the result of taking the trivial $I$-bundle
over the pillowcase $R(T^2)$ and collapsing the fibers over the four
pillowcase points.

\begin{figure}[t]
  \centering
  \includegraphics[scale=0.7]{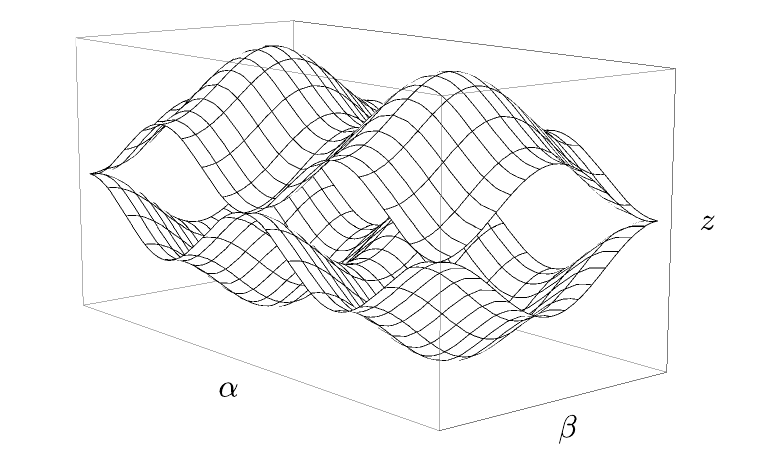}
  \caption{
    \label{fig:space-Y}
    The space $Y$, which is homeomorphic to the piece $P_3$, is
    the region between the pair of surfaces.
    The vertical faces are identified as described in Definition
    \ref{def:space-Y}.
  }
\end{figure}

\begin{theorem}
\label{theorem:space-P3}
The space $P_3$ is homeomorphic to $Y$.
Representations on the boundary of $P_3$ are abelian, and
representations on the interior of $P_3$ are nonabelian.
\end{theorem}

\begin{proof}
We first determine the fibers of the map
$q:P_3 \rightarrow R(T^2)$.
Given a conjugacy class $[\rho] \in P_3$, we can always choose a
representative of the form
\begin{align}
  \label{eqn:param-mu1}
  A &= \cos\alpha + \sin\alpha\,\qz, &
  B &= \cos\beta + \sin\beta\,\qz, &
  a &= \cos\gamma\,\qx + \sin\gamma\,\qz, &
  b &= a^{-1},
\end{align}
where $(\alpha,\beta) \in [0,2\pi] \times [0,\pi]$ and
$\gamma \in [-\pi/2,\pi/2]$.
For $(A,B) = (\pm 1, \pm 1)$, we can conjugate so as to force
$\gamma=0$.
From these considerations it follows that the fibers of $q$ are
points ($\gamma = 0$) over the four pillowcase points
$[A,B] = [\pm 1, \pm 1]$,
and intervals ($\gamma \in [-\pi/2,\pi/2]$) over all other points.
We can thus define a homeomorphism $P_3 \rightarrow Y$ by
\begin{align}
  \nonumber
  (\alpha,\, \beta,\, \gamma) \mapsto
  (\alpha,\, \beta,\, (2\gamma/\pi)(\sin^2\alpha + \sin^2\beta)),
\end{align}
where $(\alpha,\beta) \in [0,2\pi]\times [0,\pi]$, and
$\gamma \in [-\pi/2,\pi/2]$ are chosen such that equations
(\ref{eqn:param-mu1}) are satisfied.
The statement regarding abelian and nonabelian representations is
clear from  equation (\ref{eqn:param-mu1}).
\end{proof}

Our main application of Theorem \ref{theorem:space-P3} will be to
use $(\alpha,\beta,\gamma)$ as coordinates on
$P_3$, subject to the identifications
\begin{align}
  \nonumber
  (\alpha,\beta,\gamma) \sim (\alpha + 2\pi,\beta, \gamma), &&
  (\alpha,\beta,\gamma) \sim (\alpha,\beta + 2\pi,\gamma), &&
  (\alpha,\beta,\gamma) \sim (-\alpha,-\beta,-\gamma),
\end{align}
and if
$(\alpha,\beta) \in \{(0,0),\, (0,\pi),\, (\pi,0),\, (\pi,\pi)\}$,
corresponding to the four pillowcase points of $R(T^2)$, then
$(\alpha,\beta,\gamma) \sim (\alpha,\beta,0)$.
Note that $P_3$ deformation retracts onto $P_3 \cap \{\gamma = 0\}$,
which may be identified with the pillowcase $R(T^2)$.
Theorems \ref{theorem:space-P4} and \ref{theorem:space-P3} imply
Theorem \ref{theorem:intro-1} from the Introduction.

\begin{remark}
The character variety $R(T^2,2)$ is smooth away from the reducible
locus $\partial P_3$.
We note that $\partial P_3$ is homeomorphic to $T^2$; a specific
homeomorphism $T^2 \rightarrow \partial P_3$ is given by
$(\alpha,\beta) \mapsto [A,B,a,b]$, where
\begin{align}
  \nonumber
  A &= \cos\alpha + \sin\alpha\,\qz, &
  B &= \cos\beta + \sin\beta\,\qz, &
  a &= b^{-1} = \qz.
\end{align}
\end{remark}

\subsection{The unperturbed character variety $R(U_1, A_1)$ and
  Lagrangian $L_1 \subset R(T^2,2)$}
\label{sec:R}

Our next task is to determine the Lagrangian $L_1$ in $R(T^2,2)$ that
corresponds to a solid torus $U_1 = S^1 \times D^2$ containing an
unknotted arc $A_1$ connecting distinct points
$p_1, p_2 \in \partial U_1$.
We first define and describe a character variety $R(U_1,A_1)$ for
$(U_1,A_1)$.
The Lagrangian $L_1$ is then given by the image of a pullback
map $R(U_1,A_1) \rightarrow R(T^2,2)$.

\begin{definition}
\label{def:R}
We define the \emph{unperturbed} character variety $R(U_1, A_1)$ to be
the space of conjugacy classes of homomorphisms
$\rho:\pi_1(U_1 - A_1) \rightarrow SU(2)$ that map
loops around the arc $A_1$ to traceless matrices.
\end{definition}

\begin{theorem}
\label{theorem:space-R}
The space $R(U_1, A_1)$ is homeomorphic to the closed
unit disk $D^2$.
Representations on the boundary of $R(U_1,A_1)$ are abelian, and
representations on the interior of $R(U_1,A_1)$ are nonabelian.
\end{theorem}

\begin{proof}
The fundamental group of $U_1 - A_1$ is given by
\begin{align}
  \nonumber
  \pi_1(U_1 - A_1) =
  \langle A, B, a, b \mid B=1,\,b = a^{-1} \rangle,
\end{align}
where $A$ and $B$ are the longitude and meridian of the boundary of
the solid torus and $a$ and $b$ are loops in the boundary encircling
the points $p_1$ and $p_2$, respectively.

We now consider homomorphisms
$\rho:\pi_1(U_1 - A_1) \rightarrow SU(2)$
that satisfy the requirements described in
Definition \ref{def:R} for $R(U_1, A_1)$.
As usual, we use the same notation for generators of the fundamental
group and their images under $\rho$; for example, we denote $\rho(A)$
by $A$.
Given an arbitrary representative of a conjugacy class
$[\rho] \in R(U_1, A_1)$, we will argue that we can always
conjugate so as to obtain a representative of the form
\begin{align}
  \label{eqn:R-abAB}
  A &= \cos\chi + \sin\chi\,\qz, &
  B &= 1, &
  a &= b^{-1} = \cos\psi\,\qx + \sin\psi\,\qz,
\end{align}
where $(\chi,\psi) \in [0,\pi] \times [-\pi/2,\pi/2]$.
We first conjugate so the coefficients of $\qx$ and $\qy$ in $A$ are
zero and the coefficient of $\qz$ is nonnegative, and then rotate
about the $z$-axis so the coefficient of $\qy$ in $a$ is zero and the
coefficient of $\qx$ is nonnegative.
We have thus obtained a representative of the form given in equation
(\ref{eqn:R-abAB}).
If $\chi \in (0,\pi)$, then it is clear from these equations that the
representative is unique.
If $\chi \in \{0,\pi\}$ then $A=\pm 1$ and we can conjugate so that
$a = b^{-1} = \qx$, so we obtain the identifications
$(0,\psi) \sim (0,0)$ and $(\pi,\psi) \sim (\pi,0)$.
It follows that $R(U_1, A_1)$ is homeomorphic to the square
$[0,\pi] \times [-\pi/2,\pi/2]$ with the left and right edges each
collapsed to a point, and this space is homeomorphic to the closed
disk $D^2$.
The statement regarding abelian and nonabelian representations is
clear from equation (\ref{eqn:R-abAB}).
\end{proof}

Given a representation of $\pi_1(U_1 - A_1)$, we can pull back along
the inclusion
$T^2 - \{p_1,p_2\} \hookrightarrow U_1 - A_1$ to obtain a
representation of $\pi_1(T^2-\{p_1,p_2\})$.
This induces a map
$R(U_1, A_1) \rightarrow R(T^2,2)$.

\begin{definition}
We define the \emph{unperturbed} Lagrangian $L_1$ to be the image of
$R(U_1, A_1) \rightarrow R(T^2,2)$, and we denote the image in
$R(T^2,2)$ of the point in $R(U_1, A_1)$ with coordinates
$(\chi,\psi)$ by $L_1(\chi,\psi)$.
\end{definition}

The following Theorem gives an explicit description of the Lagrangian
$L_1$ in terms of the coordinates
$(\chi,\psi) \in [0,\pi] \times [-\pi/2,\pi/2]$:

\begin{theorem}
\label{theorem:Ld}
The map $R(U_1, A_1) \rightarrow R(T^2,2)$ is injective and is an
immersion on the interior of $R(U_1,A_1)$.
The image $L_1(\chi,\psi) = [A,B,a,b] \in R(T^2,2)$ of the point in
$R(U_1,A_1)$ with coordinates $(\chi,\psi)$ is given by
\begin{align}
  \nonumber
  A &= \cos\chi + \sin\chi\,\qz, &
  B &= 1, &
  a &= b^{-1} = \cos\psi\,\qx + \sin\psi\,\qz.
\end{align}
The image $L_1$ of the map lies entirely in the piece $P_3$, and the
$(\alpha,\beta,\gamma)$ coordinates of $L_1(\chi,\psi)$ are
\begin{align}
  \nonumber
  \alpha(L_1(\chi,\psi)) &= \chi, &
  \beta(L_1(\chi,\psi)) &= 0, &
  \gamma(L_1(\chi,\psi)) &= \psi.
\end{align}
Representations on the boundary of $L_1$ are abelian, and
representations on the interior of $L_1$ are nonabelian.
\end{theorem}

\begin{proof}
The representative $(A,B,a,b)$ of $L_1(\chi,\psi)$ follows directly
from equation
(\ref{eqn:R-abAB}), and the statement regarding abelian and nonabelian
representations is clear from the form of this representative.
The $(\alpha,\beta,\gamma)$ coordinates of $L_1(\chi,\psi)$ can be
read off from equation (\ref{eqn:param-mu1}).
It is clear from these expressions that the map
$R(U_1, A_1) \rightarrow R(T^2,2)$ is injective and is an immersion on
the interior of $R(U_1,A_1)$.
\end{proof}

We plot the Lagrangian $L_1$ in Figure \ref{fig:Ld}.
Theorems \ref{theorem:space-R} and \ref{theorem:Ld} imply Theorem
\ref{theorem:intro-R} from the Introduction.

\begin{figure}
  \centering
  \includegraphics[scale=0.7]{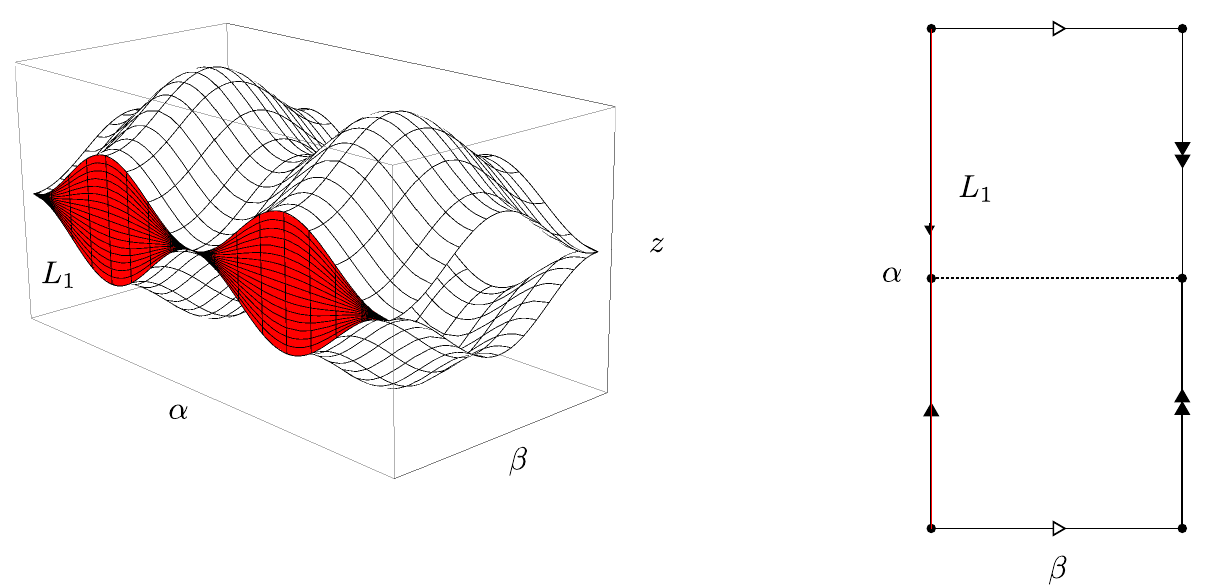}
  \caption{
    \label{fig:Ld}
    (Left)
    The unperturbed Lagrangian $L_1$ in the piece $P_3$.
    (Right)
    The intersection of the unperturbed Lagrangian $L_1$ with the
    pillowcase $P_3 \cap \{\gamma=0\}$.
  }
\end{figure}

\subsection{The perturbed character variety
  $R_\pi^\natural(U_1, A_1)$ and Lagrangian $L_1^\pi \subset R(T^2,2)$}
\label{sec:R-nat-pi}

We now want to modify the character variety $R(U_1,A_1)$ in order to
address the technical issues described in the Introduction.
Specifically, we want to (1) eliminate reducible connections, and (2)
introduce a suitable holonomy perturbation so as to render the
Chern-Simons functional Morse.
These modifications yield a perturbed character variety
$R_\pi^\natural(U_1,A_1)$.
We define a corresponding perturbed Lagrangian $L_1^\pi$ given by the
image of a pullback map
$R_\pi^\natural(U_1,A_1) \rightarrow R(T^2,2)$.

\begin{figure}[t]
  \centering
  \includegraphics[scale=0.7]{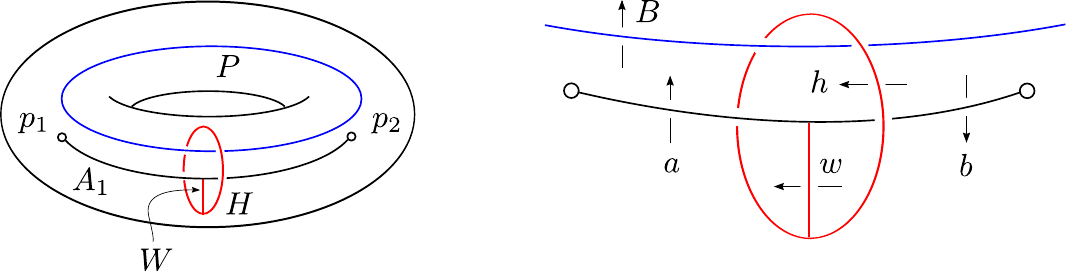}
  \caption{
    \label{fig:R-nat-pi}
    (Left)
    Solid torus $U_1$ used to
    define $R_\pi^\natural(S^1 \times D^2, A_1)$.
    Shown are the arc $A_1$, the loop $H$ and arc $W$, and the
    perturbation loop $P$.
    (Right)
    Loops $B$, $a$, $b$, $h$, and $w$.
  }
\end{figure}


We eliminate reducible connections by adding an \emph{earring}
consisting of a small loop $H$ around $A_1$ and an arc $W$
connecting $A_1$ to $H$, as shown in Figure \ref{fig:R-nat-pi}.
We require that representations take loops around $A_1$ and $H$ to
traceless matrices and loops around $W$ to $-1$.
One can show that representations satisfying these requirements must
be nonabelian, corresponding to irreducible connections.

We render the Chern-Simons functional Morse by adding a
holonomy perturbation term \cite{Kronheimer-1,Kronheimer-2}.
We choose a perturbation that vanishes outside of a small solid
torus obtained by thickening the loop $P$ shown in Figure
\ref{fig:R-nat-pi}.
The net effect of the perturbation is to impose an additional
requirement on the representations.
Specifically, letting $\lambda_P = h^{-1}A$ and $\mu_P = B$ denote the
homotopy classes of the longitude and meridian of the solid torus
obtained by thickening $P$, we require that if $\rho(\lambda_P)$ has
the form
\begin{align}
  \label{eqn:rho-lambdaP}
  \rho(\lambda_P) &=
  \cos \phi + \sin\phi\,(r_x\,\qx + r_y\,\qy + r_z\,\qz)
\end{align}
for some angle $\phi$ and some unit vector
$\hat{r} = (r_x, r_y, r_z) \in S^2$, then $\rho(\mu_P)$ must have the
form
\begin{align}
  \label{eqn:rho-muP}
  \rho(\mu_P) &=
  \cos\nu + \sin\nu\,(r_x\,\qx + r_y\,\qy + r_z\,\qz),
\end{align}
where $\nu = \epsilon f(\phi)$.
Here $\epsilon > 0$ is a small
parameter that controls the magnitude of the perturbation and
$f:\Reals \rightarrow \Reals$ is
a function such that $f(-x) = -f(x)$,
$f$ is $2\pi$-periodic, and $f(x)$ is zero if and only if $x$ is a
multiple of $\pi$.
We will usually take $f(\phi) = \sin\phi$.

We define a character variety $R_\pi^\natural(U_1,A_1)$ that includes
both of these modifications to $R(U_1,A_1)$:

\begin{definition}
\label{def:R-nat-pi}
We define the \emph{perturbed} character variety
$R_\pi^\natural(U_1, A_1)$
to be the space of conjugacy classes of homomorphisms
$\rho:\pi_1(U_1 - A_1\cup H \cup W \cup P) \rightarrow SU(2)$
that take loops around $A_1$ and $H$ to traceless matrices and loops
around $W$ to $-1$, and are such that if $\rho(\lambda_P)$ has the
form given in equation (\ref{eqn:rho-lambdaP}) then
$\rho(\lambda_P)$ must have the form given in equation
(\ref{eqn:rho-muP}).
\end{definition}

\begin{theorem}
\label{theorem:R-nat-pi}
For $\epsilon > 0$ sufficiently small, the space
$R_\pi^\natural(U_1, A_1)$ is homeomorphic to $S^2$.
All representations in $R_\pi^\natural(U_1, A_1)$ are nonabelian.
\end{theorem}

\begin{proof}
We define homotopy classes of loops $A$, $B$, $a$, $b$, and $h$ as
shown in Figure \ref{fig:R-nat-pi}, and read off relations from
 Figure \ref{fig:R-nat-pi} to obtain a presentation of
$\pi_1(U_1 - A_1\cup H \cup W \cup P)$:
\begin{align}
  \nonumber
  \pi_1(U_1 - A_1\cup H \cup W \cup P) =
  \langle A, B, a, b, h, w \mid hwaB = aBh,\,
  b = ha^{-1}w^{-1}h^{-1} \rangle.
\end{align}

We now consider homomorphisms
$\rho:\pi_1(U_1 - A_1\cup H \cup W \cup P) \rightarrow
SU(2)$ that satisfy the requirements described in Definition
\ref{def:R-nat-pi} for $R_\pi^\natural(U_1, A_1)$.
As usual, we use the same notation for generators of the fundamental
group and their images under $\rho$; for example, we denote
$\rho(A)$ by $A$.
Given an arbitrary representative of a conjugacy class
$[\rho] \in R_\pi^\natural(U_1, A_1)$, we will argue that
we can always conjugate so as to obtain a unique representative of the
form given by
\begin{align}
  \nonumber
  A &=
  h(\cos\phi + \sin\phi\,(\cos\theta\,\qx + \sin\theta\,\qy)), &
  \nonumber
  B &=
  \cos \nu + \sin\nu\,(\cos\theta\,\qx + \sin\theta\,\qy), \\
  \nonumber
  \nonumber
  a &= \qz, &
  \nonumber
  b &= -ha^{-1}h^{-1}, \\
  \nonumber
  h &= (\cos^2\nu + \sin^2\nu\sin^2\theta)^{-1/2}(
  \cos \nu\,\qx + \sin\nu \sin\theta\,\qz), &
  \nonumber
  w &= -1,
\end{align}
where $\nu = \epsilon \sin \phi$ and
$(\phi,\theta) \in [0,\pi] \times [0,2\pi]$ are spherical-polar
coordinates on $S^2$.
We first conjugate so that $a = \qz$.
Next, we rotate about the $z$-axis so that the coefficient of
$\qy$ in $h$ is zero.
After these operations have been performed, we can express $\lambda_P$
as
\begin{align}
  \nonumber
  \lambda_P = \cos\phi + \sin\phi\,(r_x\,\qx + r_y\,\qy + r_z\,\qz)
\end{align}
for some angle $\phi$ and some unit vector
$\hat{r} = (r_x, r_y, r_z) \in S^2$.
The relationship between $\lambda_P$ and $\mu_P$ described
in equations (\ref{eqn:rho-lambdaP}) and (\ref{eqn:rho-muP}) then
implies that
\begin{align}
  \nonumber
  B &= \mu_P = \cos\nu + \sin\nu\,(r_x\,\qx + r_y\,\qy + r_z\,\qz),
\end{align}
where $\nu = \epsilon \sin\phi$.
We also find that
\begin{align}
  \nonumber
  A = h\lambda_P =
  h(\cos\phi + \sin\phi\,(r_x\,\qx + r_y\,\qy + r_z\,\qz)).
\end{align}
Since $w=-1$, the relation $b = ha^{-1}w^{-1}h^{-1}$ implies that
$b = -ha^{-1}h^{-1}$,
and the relation $hwaB = aBh$ implies that $aB$ and $h$ anticommute.
Since $a = \qz$, the fact that $aB$ and $h$ anticommute implies that
$r_z = 0$, so $\hat{r} = (\cos\theta,\sin\theta,0)$
for some angle $\theta$.
The fact that $aB$ and $h$ anticommute, in conjunction with the fact
that the coefficient of $\qy$ in $h$ is zero, further implies
that $h$ must have the form
\begin{align}
  \label{eqn:h-pm}
  h =
  \pm(\cos^2\nu + \sin^2\nu\sin^2\theta)^{-1/2}
  (\cos \nu\,\qx + \sin\nu \sin\theta\,\qz).
\end{align}
In fact, we can assume that the plus sign obtains in equation
(\ref{eqn:h-pm}), since if not then we can conjugate by $\qz$
and redefine $\theta \mapsto \theta + \pi$; this operation flips the
signs of $h$ and $A$ and leaves $B$, $a$, $b$, and $w$ invariant.
We have thus obtained a representative of the desired form.
Since $a = \qz$ and the coefficient of $\qx$ in $h$ is
nonzero for $\epsilon$ sufficiently small, this representative is
unique and nonabelian.

We note that the unique representative is invariant under the
transformations
\begin{align}
  \nonumber
  (\phi,\theta) &\mapsto (\phi + 2\pi, \theta), &
  (\phi,\theta) &\mapsto (\phi, \theta + 2\pi), &
  (\phi,\theta) &\mapsto (-\phi, \theta + \pi).
\end{align}
By invariance under the first transformation we can assume that
$\phi \in [-\pi,\pi]$, by invariance under the third transformation we
can further assume that $\phi \in [0,\pi]$, and by invariance under the
second transformation we can assume that $\theta \in [0,2\pi]$.
From the equations defining the unique representative, it is clear
that the map $S^2 \rightarrow R_\pi^\natural(U_1, A_1)$,
$(\phi,\theta) \mapsto [\rho]$ is a homeomorphism, where
$(\phi,\theta)$ are spherical-polar coordinates on $S^2$.
\end{proof}

Given a representation of
$\pi_1(U_1 - A_1\cup H \cup W \cup P)$, we can pull back
along the inclusion
$U_1 - A_1\cup H \cup W \cup P \hookrightarrow T^2 - \{p_1,p_2\}$
to obtain a representation of $\pi_1(T^2 - \{p_1,p_2\})$.
This induces a map
$R_\pi^\natural(U_1, A_1) \rightarrow R(T^2,2)$.

\begin{definition}
We define the \emph{perturbed} Lagrangian $L_1^\pi$ to be the image of
$R_\pi^\natural(U_1, A_1) \rightarrow R(T^2,2)$, and we denote
the image in $R(T^2,2)$ of the point in
$R_\pi^\natural(U_1, A_1)$ with coordinates $(\phi,\theta)$
by $L_1^\pi(\phi,\theta)$.
\end{definition}

We can view the Lagrangian $L_1^\pi$ as a perturbation of $L_1$,
which we defined to be the image of
$R(U_1, A_1) \rightarrow R(T^2,2)$.
The following Theorem gives an explicit description of the Lagrangian
$L_1^\pi$ in terms of the spherical-polar coordinates
$(\phi,\theta) \in [0,\pi] \times [0,2\pi]$:

\begin{theorem}
\label{theorem:L1}
The map $R_\pi^\natural(U_1, A_1) \rightarrow R(T^2,2)$ is an
injective immersion except at the north pole $(\phi=0)$ and
south pole $(\phi=\pi)$, both of which get mapped to the same point
$(\alpha,\beta,\gamma) = (\pi/2,0,0)$ in the piece $P_3$.
The image
$L_1^\pi(\phi,\theta) = [A,B,a,b] \in R(T^2,2)$ of the point in
$R_\pi^\natural(U_1, A_1)$ with coordinates $(\phi,\theta)$ is given by
\begin{align}
  \nonumber
  A &=
  (\cos^2\nu + \sin^2\nu\sin^2\theta)^{-1/2}
  (\cos \nu\,\qx + \sin\nu \sin\theta\,\qz)
  (\cos\phi + \sin\phi\,(\cos\theta\,\qx + \sin\theta\,\qy)), \\
  \nonumber
  B &=
  \cos \nu + \sin\nu\,(\cos\theta\,\qx + \sin\theta\,\qy), \\
  \nonumber
  a &= \qz, \\
  \nonumber
  b &=
  (\cos^2\nu + \sin^2\nu\sin^2\theta)^{-1}
  (\sin2\nu \sin\theta\,\qx - (\cos^2\nu - \sin^2\nu\sin^2\theta)\qz),
\end{align}
where $\nu = \epsilon \sin \phi$ and $\epsilon > 0$ is a small control
parameter that determines the strength of the perturbation.
Points $L_1^\pi(\phi,\theta)$ with
$\phi \in (0,\pi)$, $\theta \notin \{0,\pi\}$ lie in the piece $P_4$,
and the $(\hat{a},\hat{b})$ coordinates of such points are
\begin{align}
  \nonumber
  \hat{a}(L_1^\pi(\phi,\theta)) &=
  (\sin(\phi+\nu),\, -\cos(\phi+\nu),\, 0), \\
  \nonumber
  b_x(L_1^\pi(\phi,\theta)) &=
  -(\cos^2\nu + \sin^2\nu\sin^2\theta)^{-1}
  (\cos^2\nu\cos^2\theta\sin(\phi+\nu) +
  \sin^2\theta\sin(\phi-\nu)), \\
  \nonumber
  b_y(L_1^\pi(\phi,\theta)) &=
  (\cos^2\nu + \sin^2\nu\sin^2\theta)^{-1}
  (\cos^2\nu\cos^2\theta\cos(\phi+\nu) +
  \sin^2\theta\cos(\phi-\nu)), \\
  \nonumber
  b_z(L_1^\pi(\phi,\theta)) &=
  (1/2)(\cos^2\nu + \sin^2\nu\sin^2\theta)^{-1}
  \sin(2\nu)\sin(2\theta)
\end{align}
for $\theta \in (0,\pi)$, and
\begin{align}
  \nonumber
  \hat{a}(L_1^\pi(\phi,\theta)) &=
  (-\sin(\phi+\nu),\,\cos(\phi+\nu),\,0), \\
  \nonumber
  b_x(L_1^\pi(\phi,\theta)) &=
  (\cos^2\nu + \sin^2\nu\sin^2\theta)^{-1}
  (\cos^2\nu\cos^2\theta\sin(\phi+\nu) +
  \sin^2\theta\sin(\phi-\nu)), \\
  \nonumber
  b_y(L_1^\pi(\phi,\theta)) &=
  -(\cos^2\nu + \sin^2\nu\sin^2\theta)^{-1}
  (\cos\nu^2\cos^2\theta\cos(\phi+\nu) +
  \sin^2\theta \cos(\phi-\nu)), \\
  \nonumber
  b_z(L_1^\pi(\phi,\theta)) &=
  (1/2)(\cos^2\nu + \sin^2\nu\sin^2\theta)^{-1}
  \sin(2\nu)\sin(2\theta)
\end{align}
for $\theta \in (\pi,2\pi)$.
Points $L_1^\pi(\phi,\theta)$ with $\theta \in \{0,\pi\}$ lie in the piece
$P_3$, and the $(\alpha,\beta,\gamma)$ coordinates of such points are
\begin{align}
  \nonumber
  \alpha(L_1^\pi(\phi,0)) &= \phi + \pi/2, &
  \beta(L_1^\pi(\phi,0)) &= \nu = \epsilon \sin\phi, &
  \gamma(L_1^\pi(\phi,0)) &= 0, \\
  \nonumber
  \alpha(L_1^\pi(\phi,\pi)) &= \phi - \pi/2, &
  \beta(L_1^\pi(\phi,\pi)) &= \nu = \epsilon\sin\phi, &
  \gamma(L_1^\pi(\phi,\pi)) &= 0.
\end{align}
All representations in $L_1^\pi$ are nonabelian.
\end{theorem}

\begin{proof}
The representative $(A,B,a,b)$ of
$L_1^\pi(\phi,\theta)$ follows directly from the proof of Theorem
\ref{theorem:R-nat-pi}.
The fact that all representations in $L_1^\pi$ are nonabelian follows from
the fact that $a = \qz$ and the coefficient of either $\qx$ or $\qy$
in $B$ is nonzero.
We find the $(\hat{a},\hat{b})$ coordinates for points
$L_1^\pi(\phi,\theta) \in P_4$ by conjugating
the representative of $L_1^\pi(\phi,\theta)$ so that $A$ and $B$
have the form given in equation (\ref{eqn:AB}), then reading off
$\hat{a} = (a_x, a_y, a_z)$ and
$\hat{b} = (b_x, b_y, b_z)$ from
$a = a_x\,\qx + a_y\,\qy + a_z\,\qz$ and
$b = b_x\,\qx + b_y\,\qy + b_z\,\qz$.
We find the $(\alpha,\beta,\gamma)$ coordinates for points
$L_1^\pi(\phi,\theta) \in P_3$ by substituting
$\theta = 0$ and $\theta = \pi$
into the representative of $L_1^\pi(\phi,\theta)$ and then conjugating the
resulting equations so they have the form given in equation
(\ref{eqn:param-mu1}).

We will prove that
$R_\pi^\natural(U_1, A_1) \rightarrow R(T^2,2)$
is an injective immersion on
$\phi \in (0,\pi)$, $\theta \neq \{0,\pi\}$ by
showing that the coordinates $(\phi,\theta)$ can be recovered from
certain functions defined on $R(T^2,2)$.
Define functions $h_1:R(T^2,2) \rightarrow \Complex$ and
$h_2:R(T^2,2) \cap \{\tr Aa \neq 0\} \rightarrow \Reals$ by
\begin{align}
  \nonumber
  h_1([A,B,a,b]) &= -\tr AB - (i/2)(\tr B)(\tr Aa), &
  h_2([A,B,a,b]) &= -\frac{\tr Ab}{\tr Aa}.
\end{align}
A calculation shows that
\begin{align}
  \nonumber
  h_1(L_1^\pi(\phi,\theta)) =
  \frac{2\cos\nu\sin(\phi + \nu) e^{i\theta}}
  {\sqrt{\cos^2\nu + \sin^2\nu\sin^2\theta}}.
\end{align}
We note that if $\phi \in (0,\pi)$ then
$h_1(L_1^\pi(\phi,\theta)) \neq 0$ and
$\Arg (h_1(L_1^\pi(\phi,\theta))) = \theta$.
A calculation shows that $(\tr Aa)(L_1^\pi(\phi,\theta)) \neq 0$ for
$\phi \in (0,\pi)$, $\theta \neq \{0,\pi\}$,
and for such values of $(\phi,\theta)$ we have
\begin{align}
  \nonumber
  h_2(L_1^\pi(\phi,\theta)) =
  \frac{\sin(\phi - \nu)}{\sin(\phi + \nu)} =
  \frac{\sin(\phi - \epsilon\sin\phi)}{\sin(\phi + \epsilon\sin\phi)}.
\end{align}
Define $\tilde{h}_2(\phi)$ to be the right-hand-side of this
equation.
It is straightforward to show that if $\epsilon$ is sufficiently small
then $\tilde{h}_2'(\phi) > 0$ for all $\phi \in (0,\pi)$, hence
$\tilde{h}_2:(0,\pi) \rightarrow \Reals$ is a diffeomorphism onto its
image.
We conclude that
$R_\pi^\natural(U_1, A_1) \rightarrow R(T^2,2)$ is an injective
immersion on $\phi \in (0,\pi)$, $\theta \neq \{0,\pi\}$.

We similarly prove that 
$R_\pi^\natural(U_1, A_1) \rightarrow R(T^2,2)$
is an immersion on $\phi \in (0,\pi)$, $\theta \in \{0,\pi\}$ by using
the functions
$h_1:R(T^2,2) \rightarrow \Complex$ and
$\alpha:P_3 \rightarrow \Reals$.
The statements regarding the injectivity of the map for
$\theta \in \{0,\pi\}$ are clear from the expressions for the
$(\alpha,\beta,\gamma)$ coordinates.
\end{proof}

\begin{figure}
  \centering
  \includegraphics[scale=0.7]{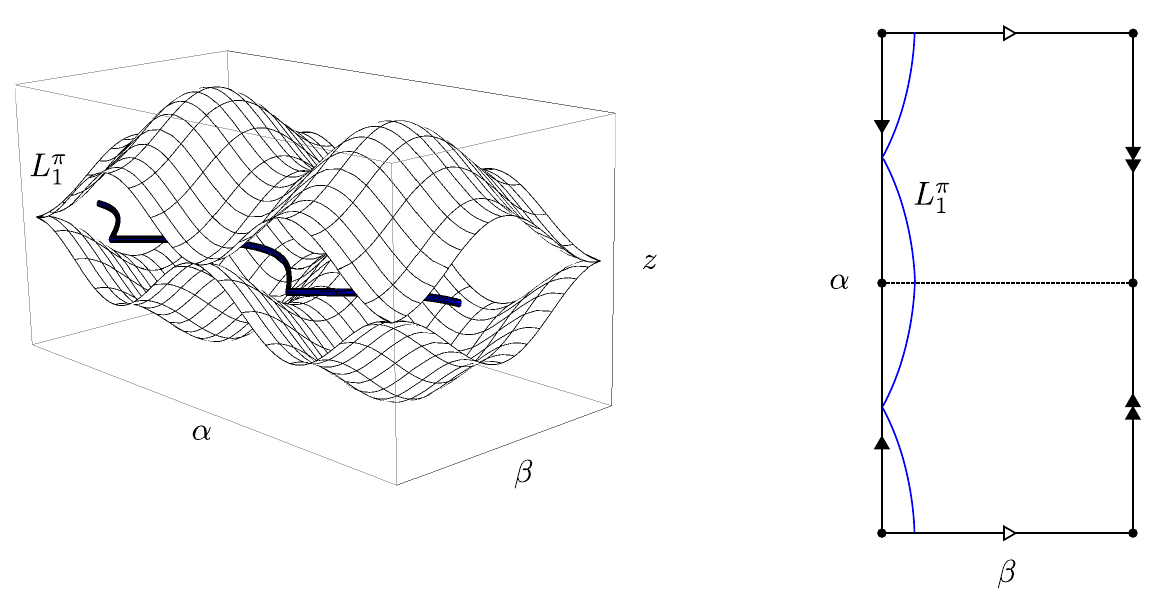}
  \caption{
    \label{fig:L1}
    (Left)
    The intersection of the perturbed Lagrangian $L_1^\pi$ with the
    piece $P_3$.
    (Right)
    The intersection of the perturbed Lagrangian $L_1^\pi$ with the
    pillowcase $P_3 \cap \{\gamma=0\}$.
  }
\end{figure}

We plot the intersection of the Lagrangian $L_1^\pi$ with the piece $P_3$
in Figure \ref{fig:L1}.
Theorems \ref{theorem:R-nat-pi} and \ref{theorem:L1} imply
Theorem \ref{theorem:intro-R-nat-pi} from the Introduction.

\begin{remark}
One can also define a character variety $R^\natural(U_1,A_1)$ that
includes the earring but not the holonomy perturbation.
It is straightforward to show that $R^\natural(U_1,A_1)$ is
homeomorphic to $S^3$ and all representations in $R^\natural(U_1,A_1)$
are nonabelian.
We will not use the character variety $R^\natural(U_1,A_1)$ here.
\end{remark}

\section{Nondegeneracy}

In this section, we adapt an argument from \cite{Abouzaid} to obtain
a simple criterion for determining when a point
$[\rho] \in R_\pi^\natural(Y,K)$ is nondegenerate; namely,
it is nondegenerate if and only if the Lagrangians
$L_1^\pi$ and $L_2$ in
$R(T^2,2)$ corresponding to the Heegaard splitting of $(Y,K)$
intersect transversely at the image of $[\rho]$ under the pullback map
$R_\pi^\natural(Y,K) \rightarrow R(T^2,2)$.
The argument relies on several results involving group cohomology and
the regularity of character varieties, which we discuss first.

\subsection{Constrained group cohomology}
\label{ssec:group-cohomology}

Consider a finitely presented group
$\Gamma = \langle S \mid R \rangle$ with generators
$S = \{s_1, \cdots, s_n\}$ and relations
$R = \{r_1, \cdots, r_m\}$.
In defining character varieties, we often want to
consider a space $X(\Gamma) \subseteq \Hom(\Gamma,SU(2))$
consisting of homomorphisms that satisfy certain constraints;
for example, we may require the homomorphisms to map certain
generators to traceless matrices.
Provided the constraints are algebraic, the space $X(\Gamma)$ has the
structure of a real algebraic variety, and we can define a
corresponding scheme ${\mathcal X}(\Gamma)$ whose set of closed points is
$X(\Gamma)$.
The group $SU(2)$ acts on the variety $X(\Gamma)$ by conjugation, and
we define the character variety $R(\Gamma)$ and character scheme
${\mathcal R}(\Gamma)$ to be the GIT quotients
$X(\Gamma)//SU(2)$ and ${\mathcal X}(\Gamma)//SU(2)$.
Generalizing a result due to Weil for the unconstrained case
\cite{Weil},
we have that the Zariski tangent space $T_{[\rho]}{\mathcal R}(\Gamma)$ of
the character scheme ${\mathcal R}(\Gamma)$ at a closed point $[\rho]$ can
be identified with the constrained group cohomology
$H_c^1(\Gamma; \Ad \rho)$, which we define here.

Roughly speaking, the constrained group cohomology
$H_c^1(\Gamma; \Ad \rho)$ describes infinitesimal deformations of
homomorphisms $\rho:\Gamma \rightarrow SU(2)$ that satisfy the
relevant constraints, modulo deformations that can be obtained by the
conjugation action of $SU(2)$.
The precise definition of $H_c^1(\Gamma; \Ad \rho)$ that we will use
is as follows.
Define a function
$F_r:\Hom(\langle S \rangle,SU(2)) \rightarrow SU(2)^m$,
where $\langle S \rangle$ is the free group on $S$, by
\begin{align}
  \nonumber
  F_r(\rho) = (\rho(r_1),\cdots,\rho(r_m)).
\end{align}
Thus $F_r(\rho) = (1,\cdots, 1)$ if and only if
$\rho:\langle S \rangle \rightarrow SU(2)$ preserves all the
relations in $R$ and thus descends to a homomorphism
$\rho:\Gamma \rightarrow SU(2)$.
Given a homomorphism $\rho:\Gamma \rightarrow SU(2)$ and a function
$\eta:S \rightarrow \mathfrak{g}$, where $\mathfrak{g}$ is the Lie
algebra of $SU(2)$, define a homomorphism
$\rho_t:\langle S \rangle \rightarrow SU(2)$ such that
\begin{align}
  \nonumber
  \rho_t(s_k) = e^{t\eta(s_k)}\rho(s_k).
\end{align}
Note that we can view $\eta$ as a vector in $\mathfrak{g}^{\oplus n}$.
We define a linear map
$c_r:\mathfrak{g}^{\oplus n} \rightarrow \mathfrak{g}^{\oplus m}$ by
\begin{align}
  \nonumber
  c_r(\eta) = \frac{d}{dt}F_r(\rho_t)|_{t=0}.
\end{align}
Thus $c_r(\eta) = 0$ if and only if $\eta$ describes an infinitesimal
deformation of $\rho$ that is a homomorphism
$\Gamma \rightarrow SU(2)$.

Homomorphisms $\Gamma \rightarrow SU(2)$ that represent points
in a character variety may be required to satisfy certain
constraints; for example, that they take particular generators to
traceless matrices.
Define a function
$F_c:\Hom(\langle S \rangle,SU(2)) \rightarrow \Reals^q$ such that
$F_c(\rho) = 0$ if
and only if $\rho$ satisfies these constraints; for example, if we
require that $\rho$ take the generator $s_1$ to a traceless matrix, we
would define $F_c:\Hom(\langle S\rangle,SU(2)) \rightarrow \Reals$ by
\begin{align}
  \nonumber
  F_c(\rho) = \tr(\rho(s_1)).
\end{align}
We define a linear map
$c_c:\mathfrak{g}^{\oplus n} \rightarrow \Reals^q$ by
\begin{align}
  \nonumber
  c_c(\eta) = \frac{d}{dt}F_c(\rho_t)|_{t=0}.
\end{align}
Thus $c_c(\eta) = 0$ if and only if $\eta$ describes an infinitesimal
deformation of $\rho$ that satisfies the constraints.

We now combine the linear maps for the relations and constraints to
obtain a linear map
$c:\mathfrak{g}^{\oplus n} \rightarrow
\mathfrak{g}^{\oplus m} \oplus \Reals^q$,
$c(\eta) = (c_r(\eta), c_c(\eta))$.
Given a homomorphism $\rho:\Gamma \rightarrow SU(2)$ that satisfies
the constraints, we define the space of 1-cocycles to be
\begin{align}
  \nonumber
  Z_c^1(\Gamma; \Ad \rho) &= \ker c,
\end{align}
so a vector $\eta \in \mathfrak{g}^n$ is a 1-cocycle if and only if it
describes an infinitesimal deformation of $\rho$ that is a
homomorphism that preserves the constraints.
We define the space of 1-coboundaries to be infinitesimal deformations
of $\rho$ that are obtained via the conjugation action of $SU(2)$:
\begin{align}
  \nonumber
  B_c^1(\Gamma; \Ad \rho) &=
  \{\eta:S \rightarrow \mathfrak{g} \mid
  \textup{there exists $u \in \mathfrak{g}$ such that
    $\eta(s_k) = u - \Ad_{\rho(s_k)}u$ for $k=1,\cdots, n$}\}.
\end{align}
Here $\Ad_g u := g u g^{-1}$ for $g \in SU(2)$ and
$u \in \mathfrak{g}$.
We define the constrained group cohomology $H_c^1(\Gamma; \Ad\rho)$ to
be
\begin{align}
  \nonumber
  H_c^1(\Gamma;\Ad\rho) =
  Z_c^1(\Gamma;\Ad\rho)/B_c^1(\Gamma;\Ad\rho).
\end{align}

\subsection{Regularity}

We define the \emph{local dimension} $\dim_{[\rho]} R(\Gamma)$ of
$R(\Gamma)$ at $[\rho] \in R(\Gamma)$ to be the maximal
dimension of the irreducible components of $R(\Gamma)$ containing
$[\rho]$.
We say that a point $[\rho]$ of a character variety
$R(\Gamma)$ is \emph{regular} if
\begin{align*}
  \dim_{[\rho]} R(\Gamma) = \dim H_c^1(\Gamma; \Ad \rho).
\end{align*}
We define $R'(\Gamma)$ to be the set of regular points of
$R(\Gamma)$.
The set $R'(\Gamma)$ has the structure of a smooth manifold, and the
smooth tangent space at a point $[\rho] \in R'(\Gamma)$ is
$T_{[\rho]} R(\Gamma) = H_c^1(\Gamma; \Ad \rho)$.
We will prove theorems that describe the regular
points of the character varieties $R(U_1,A_1)$,
$R_\pi^\natural(U_1,A_1)$, and $R(T^2,2)$:

\begin{theorem}
\label{theorem:reg-R}
The character variety $R(U_1,A_1)$ is regular at all points
represented by nonabelian homomorphisms.
\end{theorem}

\begin{proof}
Using results from the proof of Theorem \ref{theorem:space-R}, we
find that we can take the set of generators of the fundamental group
$\Gamma$ to be $S = \{A, a\}$, with no relations, and we can take
the constraint function
$F_c:\Hom(\langle S \rangle, SU(2)) \rightarrow \Reals$ to be
\begin{align}
  \nonumber
  F_c(\rho) = \tr(\rho(a)).
\end{align}
Using the expressions for the
homomorphisms $\rho:\Gamma \rightarrow SU(2)$ given in the proof of
Theorem \ref{theorem:space-R}, we obtain a linear map
$c:\Reals^6 \rightarrow \Reals$.
A straightforward calculation shows that
$\dim H_c^1(\Gamma; \Ad \rho) = \dim R(U_1,A_1) = 2$ for all
$[\rho] \in R(U_1,A_1)$ such that $\rho$ is nonabelian.
\end{proof}

We would next like to determine the regular points of the
perturbed character variety $R_\pi^\natural(U_1,A_1)$, but
there are two difficulties that must be overcome.
The first difficulty involves the function $f(\phi)$ that defines the
perturbation.
Recall that points $[\rho] \in R_\pi^\natural(U_1,A_1)$ are
constrained by the requirement that if $\rho(\lambda_P)$ has the form
$\rho(\lambda_P) =
\cos\phi + \sin\phi\,(r_x\,\qx + r_y\,\qy + r_z\,\qz)$,
then $\rho(\mu_P)$ must have the form
$\rho(\mu_P) = \cos\nu + \sin\nu\,(r_x\,\qx + r_y\,\qy + r_z\,\qz)$,
where $\nu = \epsilon f(\phi)$.
In order to give $R_\pi^\natural(U_1,A_1)$ the structure of a real
algebraic variety, and to define the corresponding character scheme,
this constraint must be algebraic.
We will therefore choose $f(\phi)$ to be
\begin{align}
  \label{eqn:f-alg}
  f(\phi) = \frac{1}{\epsilon}\sin^{-1}(\epsilon \sin\phi).
\end{align}
Then the constraint on $\rho$ becomes
\begin{align}
  \label{eqn:alg-constraint}
  \epsilon\tr(\rho(\lambda_P)\qx) &= \tr(\rho(\mu_P)\qx), &
  \epsilon\tr(\rho(\lambda_P)\qy) &= \tr(\rho(\mu_P)\qy), &
  \epsilon\tr(\rho(\lambda_P)\qz) &= \tr(\rho(\mu_P)\qz).
\end{align}
In fact, the constraint given in equation (\ref{eqn:alg-constraint})
yields a variety with two connected components, one with
$\rho(\mu_P)$ near 1 and one with $\rho(\mu_P)$ near -1, and only the
first component corresponds to $R_\pi^\natural(U_1,A_1)$.
To calculate the constrained group cohomology, however, we consider
only infinitesimal deformations of homomorphisms, hence the
extraneous second component is irrelevant.

A second difficulty in determining the regular points of
$R_\pi^\natural(U_1,A_1)$
is that a direct calculation of the constrained group
cohomology for
$R_\pi^\natural(U_1,A_1)$ does not appear to be practical,
because the perturbed representations, as described in Theorem
\ref{theorem:R-nat-pi}, are rather complicated.
Instead, we will apply the following theorem, which simplifies the
necessary calculations by allowing us to extrapolate from
unperturbed representations:

\begin{theorem}
\label{theorem:c-taylor}
Consider a character variety $R_\epsilon(\Gamma)$ in which the
homomorphisms are required to satisfy an algebraic
constraint that
depends on a control parameter $\epsilon \in \Reals$.
Given a homomorphism  $\rho_\epsilon:\Gamma \rightarrow SU(2)$
representing a point $[\rho_\epsilon] \in R_\epsilon(\Gamma)$, let
$c_\epsilon:\mathfrak{g}^{\oplus n} \rightarrow
\mathfrak{g}^{\oplus m} \oplus \Reals^q$ denote the corresponding
linear map used to define the constrained group cohomology.
Define
$c_0, c_1:\mathfrak{g}^{\oplus n} \rightarrow
\mathfrak{g}^{\oplus m} \oplus \Reals^q$ such that
$c_\epsilon = c_0 + \epsilon c_1 + \cdots$.
The following string of inequalities holds for $\epsilon > 0$
sufficiently small:
\begin{align}
  \label{eqn:dim-cokernel-bound}
  \dim Z_c^1(\Gamma; \Ad \rho_\epsilon) \leq
  \dim (\ker c_0 \cap \ker c_1) +
  \dim (c_1(\ker c_0) \cap \im c_0) \leq
  \dim Z_c^1(\Gamma; \Ad \rho_0).
\end{align}
\end{theorem}

\begin{proof}
Since the dimension of the Zariski tangent space is upper
semi-continuous, for $\epsilon > 0$ sufficiently small we have that
\begin{align}
  \nonumber
  \dim (\ker c_\epsilon) = \dim Z_c^1(\Gamma; \Ad \rho_\epsilon) \leq
  \dim Z_c^1(\Gamma; \Ad \rho_0) = \dim (\ker c_0).
\end{align}
Thus any vector $w_\epsilon \in \ker c_\epsilon$ must have the form
$w_\epsilon = w_0 + \epsilon w_1 + \cdots$, where
\begin{align}
  \label{eqn:c-taylor}
  c_\epsilon(w_\epsilon) = c_0(w_0) +
  \epsilon(c_0(w_1) + c_1(w_0)) + \cdots = 0.
\end{align}
The space of vectors $w_0 \in \mathfrak{g}^{\oplus n}$ that
satisfy equation (\ref{eqn:c-taylor}) up to first order in $\epsilon$
is
\begin{align}
  \nonumber
  V = \{w_0 \in \ker c_0 \mid c_1(w_0) \in \im c_0\}.
\end{align}
Since $\ker c_\epsilon = Z_c^1(\Gamma; \Ad \rho_\epsilon)$ is the space
of vectors that satisfies equation (\ref{eqn:c-taylor}) to all orders
in $\epsilon$, it follows that
$Z_c^1(\Gamma; \Ad \rho_\epsilon) \subseteq V \subseteq
\ker c_0 = Z_c^1(\Gamma; \Ad \rho_0)$, and we have the string
of inequalities
\begin{align}
  \nonumber
  \dim Z_c^1(\Gamma; \Ad \rho_\epsilon) \leq \dim V \leq
  \dim Z_c^1(\Gamma; \Ad \rho_0).
\end{align}
Equation (\ref{eqn:dim-cokernel-bound}) now follows from the fact that
\begin{align}
  \nonumber
  \dim V =
  \dim (\ker c_0 \cap \ker c_1) +
  \dim (c_1(\ker c_0) \cap \im c_0).
\end{align}
\end{proof}

\begin{example}
Take $\Gamma = \Ints$, and consider the character varieties
$R_\epsilon^i(\Gamma)$ for $i=1,2,3$ with constraint functions
$F_c^i:\Hom(\Gamma,SU(2)) \rightarrow \Reals$ given by
\begin{align}
  \nonumber
  F_c^1(\rho) &= \epsilon \tr \rho(1), &
  F_c^2(\rho) &= \epsilon (\tr \rho(1))^2, &
  F_c^3(\rho) &= \epsilon (\tr \rho(1))^2 + \epsilon^2 \tr \rho(1).
\end{align}
The character varieties are given by
\begin{align}
  \nonumber
  R_\epsilon^1(\Gamma) = R_\epsilon^2(\Gamma) = R_\epsilon^3(\Gamma) =
  \left\{
  \begin{array}{ll}
    S^2 & \quad \mbox{if $\epsilon \neq 0$,} \\
    S^3 & \quad \mbox{if $\epsilon = 0$.} \\
  \end{array}
  \right.
\end{align}
Consider the homomorphism $\rho_\epsilon:\Ints \rightarrow SU(2)$,
$\rho_\epsilon(1) = \qz$.
Then $\dim Z_c^1(\Gamma; \Ad \rho_\epsilon)$ and
$\dim (\ker c_0 \cap \ker c_1) + \dim (c_1(\ker c_0) \cap \im c_0)$
are given by
\begin{align}
  \nonumber
  \begin{array}{cccc}
    {} & F_c^1 & F_c^2 & F_c^3 \\
    \dim Z_c^1(\Gamma; \Ad \rho_\epsilon) & 2 & 3 & 2 \\
    \dim (\ker c_0 \cap \ker c_1) + \dim (c_1(\ker c_0) \cap \im c_0)
    & 2 & 3 & 3
  \end{array}
\end{align}
From the expressions for $\dim Z_c^1(\Gamma; \Ad \rho_\epsilon)$, we
find that for $\epsilon \neq 0$ the character schemes
${\mathcal R}_\epsilon^1(\Gamma)$ and ${\mathcal R}_\epsilon^3(\Gamma)$ are
reduced, and the character scheme
${\mathcal R}_\epsilon^2(\Gamma)$ is not reduced.
We can use Theorem \ref{theorem:c-taylor} to show that
${\mathcal R}_\epsilon^1(\Gamma)$ is reduced, but not that
${\mathcal R}_\epsilon^3(\Gamma)$ is reduced.
\end{example}

\begin{theorem}
\label{theorem:reg-R-nat-pi}
The character variety $R_\pi^\natural(U_1,A_1)$ is
regular everywhere.
\end{theorem}

\begin{proof}
Using results from the proof of Theorem \ref{theorem:R-nat-pi}, we
find that we can take the set of generators for the fundamental group
$\Gamma$ to be $S = \{a, A, B, h\}$, the relations function
$F_r:\Hom(\langle S \rangle,SU(2)) \rightarrow SU(2)$ to be
\begin{align}
  \nonumber
  F_r(\rho) = -\rho([h,aB]),
\end{align}
and the constraint function
$F_c:\Hom(\langle S \rangle,SU(2)) \rightarrow \Reals^6$ to be
\begin{align}
  \nonumber
  F_c(\rho) &=
  (\tr(\rho(a)),\,\tr(\rho(ha^{-1}h^{-1})),\,\tr(\rho(h)),\,
  f(\rho,\qx),\,f(\rho,\qy),\,f(\rho,\qz)),
\end{align}
where
\begin{align}
  \nonumber
  f(\rho,q) &= \epsilon \tr(\rho(h^{-1}A)q) - \tr(\rho(B)q).
\end{align}
Using the expressions for the homomorphisms
$\rho_\epsilon:\Gamma \rightarrow SU(2)$ given in the proof of
Theorem \ref{theorem:R-nat-pi}, we obtain a linear map
$c_\epsilon:\Reals^{12} \rightarrow \Reals^9$.
We now apply Theorem \ref{theorem:c-taylor}.
A straightforward, but rather lengthy, calculation shows that
$\dim (\ker c_0 \cap \ker c_1) +
\dim (c_1(\ker c_0) \cap \im c_0) = 5$ for all homomorphisms
representing points in $R_\pi^\natural(U_1,A_1)$.
Since these homomorphisms are all nonabelian, we conclude that
$\dim H_c^1(\Gamma; \Ad \rho) =
\dim R_\pi^\natural(U_1,A_1) = 2$ for all
$[\rho] \in R_\pi^\natural(U_1,A_1)$, and thus 
$R_\pi^\natural(U_1,A_1)$ is regular everywhere.
\end{proof}

\begin{theorem}
\label{theorem:reg-R2}
The image $L_1^\pi$ of the immersion
$R_\pi^\natural(U_1, A_1) \rightarrow R(T^2,2)$ lies in the regular
locus of $R(T^2,2)$.
\end{theorem}

\begin{proof}
Using results from Section \ref{sec:R2}, we find that we can take the
set of generators for the fundamental group $\Gamma$ to be
$S = \{a, A, B\}$, with no relations, and we can take the constraint
function $F_c:\Hom(\langle S \rangle, SU(2)) \rightarrow \Reals^2$ to
be
\begin{align}
  \nonumber
  F_c(\rho) = (\tr(\rho(a)),\, \tr(\rho(ABA^{-1}B^{-1}a))).
\end{align}
Using results from the proof of Theorem
\ref{theorem:R-nat-pi}, we obtain a linear map
$c_\epsilon:\Reals^9 \rightarrow \Reals^2$ for homomorphisms
representing points in $L_1^\pi$.
A straightforward, but rather lengthy, calculation shows that
$\dim (\ker c_0 \cap \ker c_1) +
\dim (c_1(\ker c_0) \cap \im c_0) = 7$ for all homomorphisms
representing points in $L_1^\pi$.
Since these homomorphisms are all nonabelian, we conclude that
$\dim H_c^1(\Gamma; \Ad \rho) = \dim R(T^2,2) = 4$.
\end{proof}

We conjecture that $R(T^2,2)$ is in fact regular at all points
represented by nonabelian homomorphisms, but Theorem
\ref{theorem:reg-R2} will suffice for our purposes.

\subsection{Transversality}

We are now ready to prove our key result that relates nondegeneracy to
transversality.
Recall that we defined the Lagrangian $L_2$ to be the image of
$R(U_2,A_2) \rightarrow R(T^2,2)$.
If $R(U_2,A_2) \rightarrow R(T^2,2)$ is injective, and
$[\rho] \in L_1^\pi \cap L_2 \subset R(T^2,2)$ is not the double-point of
$L_1^\pi$, then by Corollary \ref{cor:glue}
the point $[\rho]$ is the image of a unique point in
$R_\pi^\natural(Y,K)$ under the pullback map
$R_\pi^\natural(Y,K) \rightarrow R(T^2,2)$.
The following is a restatement of Theorem
\ref{theorem:intro-transverse} from the Introduction:

\begin{theorem}
Suppose $R(U_2,A_2) \rightarrow R(T^2,2)$ is an injective
immersion and $[\rho] \in L_1^\pi \cap L_2$
is the image of a regular point of $R(U_2,A_2)$ and is not the
double-point of $L_1^\pi$.
Then the unique preimage of $[\rho]$
under the pullback map $R_\pi^\natural(Y,K) \rightarrow R(T^2,2)$ is
nondegenerate if and only if the intersection of $L_1^\pi$ with $L_2$ at
$[\rho] \in L_1^\pi \cap L_2$ is transverse.
\end{theorem}

\begin{proof}
We introduce the notation
$K' = K \cup W \cup H \cup P$,
$Y' = Y - K'$, $U_i' = U_i - K'$, and
$\Sigma' = T^2 - \{p_1,p_2\}$.
We have the following Mayer-Vietoris sequence:
\begin{eqnarray}
  \nonumber
  \begin{tikzcd}[column sep=0.7cm]
    \cdots \arrow{r} &
    H_c^0(\Sigma'; \Ad \rho) \arrow{r} &
    H_c^1(Y'; \Ad \rho) \arrow{r} &
    H_c^1(U_1'; \Ad \rho) \oplus H_c^1(U_2'; \Ad\rho) \arrow{r} &
    H_c^1(\Sigma'; \Ad\rho) \arrow{r} &
    \cdots.
  \end{tikzcd}
\end{eqnarray}
Here $H_c^0(\Sigma'; \Ad \rho)$ is
\begin{align}
  \nonumber
  H_c^0(\Sigma'; \Ad \rho) =
  \{x \in \mathfrak{g} \mid
  \textup{$[\rho(\lambda), x] = 0$ for all
    $\lambda \in \pi_1(\Sigma')$}\},
\end{align}
and $H_c^1(Y'; \Ad \rho)$, $H_c^1(U_1'; \Ad \rho)$,
$H_c^1(U_2'; \Ad\rho)$, $H_c^1(\Sigma'; \Ad\rho)$ are the constrained
group cohomology for the character varieties
$R_\pi^\natural(Y,K)$, $R_\pi^\natural(U_1,A_1)$,
$R(U_2,A_2)$, and $R(T^2,2)$, respectively.
For notational simplicity, we are using $\rho$ to denote a
homomorphism representing a point in $R_\pi^\natural(Y,K)$, as
well as its pullbacks to homomorphisms representing points in
$R_\pi^\natural(U_1,A_1)$, $R(U_2,A_2)$, and $R(T^2,2)$.
From Theorem \ref{theorem:L1} we have that all points in $L_1^\pi$ are
represented by nonabelian homomorphisms, thus
$H_c^0(\Sigma'; \Ad \rho) = 0$.
From Theorems \ref{theorem:reg-R-nat-pi} and \ref{theorem:reg-R2}, we
have the identifications
\begin{align}
  \nonumber
  H_c^1(U_1';\Ad \rho) &= T_{[\rho]}R_\pi^\natural(U_1,A_1), &
  H_c^1(\Sigma'; \Ad\rho) &= T_{[\rho]}R(T^2,2).
\end{align}
Since we have assumed that
$[\rho] \in R(U_2,A_2)$ is regular, we have the
identification
\begin{align}
  \nonumber
  H_c^1(U_2';\Ad \rho) &= T_{[\rho]}R(U_2,A_2).
\end{align}
By Theorem \ref{theorem:L1}, the map
$R_\pi^\natural(U_1,A_1) \rightarrow R(T^2,2)$ is an
immersion (with image $L_1^\pi$), and we have assumed that
$R(U_2,A_2) \rightarrow R(T^2,2)$ is an immersion (with image $L_2$),
so we can identify
\begin{align}
  \nonumber
  T_{[\rho]}R_\pi^\natural(U_1,A_1) &= T_{[\rho]}L_1^\pi, &
  T_{[\rho]}R(U_2,A_2) &= T_{[\rho]}L_2.
\end{align}
We conclude that the constrained group cohomology
$H_c^1(Y'; \Ad \rho)$ is given by
\begin{align}
  \nonumber
  H_c^1(Y'; \Ad \rho) = T_{[\rho]}{\mathcal R}_\pi^\natural(Y,K) = 
  T_{[\rho]}L_1^\pi \cap T_{[\rho]}L_2.
\end{align}
The constrained  group cohomology $H_c^1(Y'; \Ad \rho)$ is zero if and
only if $[\rho]$ is nondegenerate (see \cite{Donaldson} Section 2.5.4).
Thus $[\rho]$ is nondegenerate if and only if
$L_1^\pi$ intersects $L_2$ transversely at $[\rho]$.
\end{proof}

\begin{example}
Consider the algebraic functions $f,g:\Reals \rightarrow \Reals$,
$f(x) = x^2$, $g(x) = x^3$.
The schemes corresponding to the critical loci of $f$ and $g$ are
$\Spec F = \{(0)\}$ and $\Spec G = \{(x)\}$, where
\begin{align}
  \nonumber
  F &= \Reals[x]/(f'(x)) = \Reals, &
  G &= \Reals[x]/(g'(x)) = \Reals[x]/(x^2).
\end{align}
The fact that 0 is a nondegenerate critical point of $f$, but a
degenerate critical point of $g$, is reflected in the fact that $F$ is
reduced, but $G$ is nonreduced, which in turn is reflected in the fact
that $T_{(0)}\Spec F = 0$, but $T_{(x)}\Spec G = \Reals$.
\end{example}

Since nondegeneracy is a stable property, for sufficiently small
$\epsilon > 0$ we can use the function
$f(\phi) = \sin\phi$ to define the perturbation, rather than
the function $f(\phi)$ given in equation (\ref{eqn:f-alg}).

\section{The group $\MCG_2(T^2)$ and its action on $R(T^2,2)$}
\label{sec:mcg}

An important property of the character variety $R(T^2,2)$ is that it
admits an action of the mapping class group $\MCG_2(T^2)$.
Here we describe the group $\MCG_2(T^2)$ and its action on $R(T^2,2$).

\subsection{The mapping class group $\MCG_2(T^2)$}

\begin{definition}
Given a surface $S$ and $n$ distinct marked points
$p_1, \cdots, p_n \in S$, we define
the \emph{mapping class group $\MCG_n(S)$} to be the group
of isotopy classes of orientation-preserving homeomorphisms of $S$
that fix $\{p_1, \cdots, p_n\}$ as a set.
\end{definition}

Presentations for mapping class groups are described in
\cite{Cattabriga,Gervais,Labruere}.
The mapping class group $\MCG_2(T^2)$ for the twice-punctured torus is
generated by Dehn twists
$T_a$, $T_A$, $T_b$, and $T_B$ around the simple closed curves
$a$, $A$, $b$, and $B$ shown in Figure \ref{fig:pmcg2}, together with
a $\pi$-rotation $\omega$ of the square shown in Figure \ref{fig:pmcg2}.
The mapping class group $\MCG(T^2) := \MCG_0(T^2)$ for the unpunctured
torus is generated by the Dehn twists $T_a$ and $T_b$.

\begin{figure}
  \centering
  \includegraphics[scale=0.7]{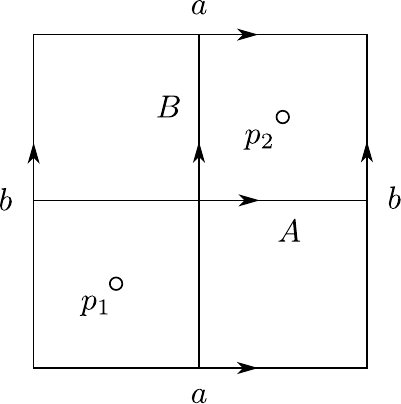}
  \caption{
    \label{fig:pmcg2}
    Cycles $a$, $A$, $b$, and $B$ corresponding to generators
    $T_a$, $T_A$, $T_b$, and $T_B$ of $\MCG_2(T^2)$.
  }
\end{figure}

It is useful to relate the mapping class groups $\MCG_2(T^2)$ and
$\MCG(T^2)$ to the braid group $B_2(T^2)$, which we define as follows:

\begin{definition}
Given a surface $S$, we define the
\emph{configuration space for ordered points $\Conf_n'(S)$} to be
the space
$\{(p_1, \cdots, p_n) \in S^n \mid \textup{$p_i \neq p_j$ if
$i\neq j$}\}$.
We define the
\emph{configuration space for unordered points $\Conf_n(S)$} to be
the space $\Conf_n'(S)/\Sigma_n$, where the fundamental group on $n$
letters $\Sigma_n$ acts on $\Conf_n'(S)$ by permutation.
\end{definition}

\begin{definition}
Given a surface $S$ and $n$ distinct marked points
$p_1, \cdots, p_n \in S$, we define
the \emph{braid group $B_n(S)$} to be the fundamental group of
$\Conf_n(S)$ with base point $[(p_1, \cdots, p_n)]$.
\end{definition}

Presentations for braid groups are described in \cite{Bellingeri}.
The braid group $B_2(T^2)$ for the twice-punctured torus is generated
by braids $\alpha_i$ and $\beta_i$ for $i=1,2$ that drag marked
the point $p_i$ rightward and upward around a cycle, together with a
braid $\sigma$ that interchanges the marked points $p_1$ and $p_2$ via
a counterclockwise $\pi$-rotation.
These generators are depicted in Figure \ref{fig:b2}.

\begin{figure}
  \centering
  \includegraphics[scale=0.7]{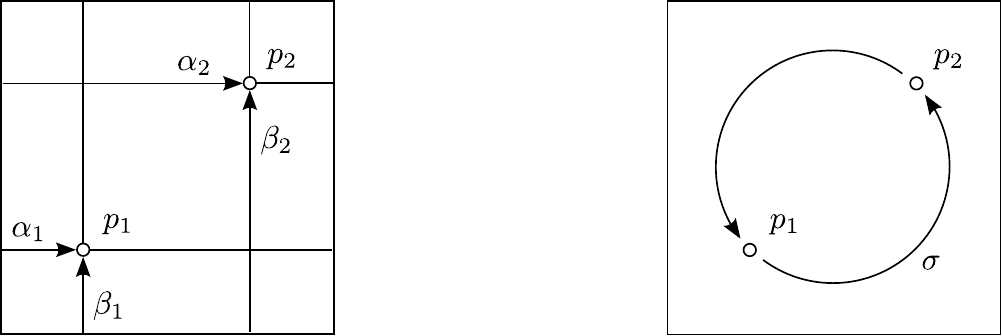}
  \caption{
    \label{fig:b2}
    (Left)
    Generators $\alpha_1$, $\alpha_2$, $\beta_1$, and $\beta_2$ of
    $B_2(T^2)$.
    (Right)
    Generator $\sigma$ of $B_2(T^2)$.
  }
\end{figure}

The braid group $B_2(T^2)$ and the mapping class groups $\MCG_2(T^2)$
and $\MCG(T^2)$ are related by the Birman exact sequence
\cite{Birman}:
\begin{eqnarray}
  \label{seq:birman}
  \begin{tikzcd}
    1 \arrow{r} &
    \pi_1(\Homeo_0(T^2)) \arrow{r} &
    B_2(T^2) \arrow{r}{p} &
    \MCG_2(T^2) \arrow{r}{g} &
    \MCG(T^2) \arrow{r} &
    1.
  \end{tikzcd}
\end{eqnarray}
Here $\Homeo_0(T^2)$ is the group of orientation-preserving
homeomorphisms of $T^2$ that are isotopic to the identity.
The group $\Homeo_0(T^2)$ deformation retracts onto
$T^2$ \cite{Hamstrom}, so $\pi_1(\Homeo_0(T^2)) = \pi_1(T^2) = \Ints^2$.
The two free abelian generators of $\pi_1(\Homeo_0(T^2))$ can be
identified with the elements $\alpha_1\alpha_2$ and $\beta_1\beta_2$
of $B_2(T^2)$ under the injection
$\pi_1(\Homeo_0(T^2)) \rightarrow B_2(T^2)$.
The push homomorphism $p:B_2(T^2) \rightarrow \MCG_2(T^2)$ is
given by
\begin{align}
  \nonumber
  p(\alpha_1) &= p(\alpha_2)^{-1} = T_a T_A^{-1}, &
  p(\beta_1) &= p(\beta_2)^{-1} = T_b T_B^{-1}, &
  p(\sigma) &= (T_a T_b^{-1} T_a)^2 \omega.
\end{align}
The forgetful homomorphism $g:\MCG_2(T^2) \rightarrow \MCG(T^2)$ is
given by
\begin{align}
  \nonumber
  g(T_a) = g(T_A) = T_a, &&
  g(T_b) = g(T_B) = T_b, &&
  g(\omega) = (T_a T_b^{-1} T_a)^2.
\end{align}
In what follows we will use the generators of $B_2(T^2)$ to also
denote their images in $\MCG_2(T^2)$ under
$p:B_2(T^2) \rightarrow \MCG_2(T^2)$.

We will use elements of the group $\MCG_2(T^2)$ to describe gluing
data for constructing $(1,1)$-knots.
By definition, a $(1,1)$-knot $K$ in a lens space $Y$ can be obtained
by gluing together two copies of a solid torus containing an unknotted
arc via a homeomorphism that represents an element $f \in \MCG_2(T^2)$.
The Birman sequence is useful for understanding the relationship
between elements $f \in \MCG_2(T^2)$ and the corresponding pairs
$(Y,K)$.
The lens space $Y$ can be recovered from the image of $f$ under
$g:\MCG_2(T^2) \rightarrow \MCG(T^2)$, so this map can be viewed as
forgetting the part of the gluing data used to construct the knot and
preserving the part of the data used to construct the lens space.
If we multiply $f$ by an element in the image of the map
$p:B_2(T^2) \rightarrow \MCG_2(T^2)$, the resulting element
$f' \in \MCG_2(T^2)$ yields a pair $(Y,K')$ consisting of a
potentially different knot $K'$ in the same lens space $Y$.
The braid group $B_2(T^2)$ is thus useful for constructing different
knots in a fixed lens space.

\subsection{The action of $\MCG_2(T^2)$ on $R(T^2,2)$}
\label{ssec:action}

We will define an action of the group $\MCG_2(T^2)$ on the character
variety $R(T^2,2)$ via a homomorphism from
$\MCG_2(T^2)$ to $\Out(\pi_1(T^2 - \{p_1, p_2\}))$, the group of outer
automorphisms of $\pi_1(T^2 - \{p_1, p_2\})$.
In general, we define a group homomorphism from $\MCG_n(T^2)$ to
$\Out(\pi_1(T^2 - \{p_1, \cdots, p_n\}))$, the group of outer
automorphisms of $\pi_1(T^2 - \{p_1, \cdots, p_n\})$, as follows.
Define $X = T^2 - \{p_1, \cdots, p_n\}$.
Choose a base point $x_0 \in X$ and consider the fundamental group
$\pi_1(X,x_0)$.
Given an element $[\phi] \in \MCG_n(X)$ represented by a homeomorphism
$\phi:X \rightarrow X$, there is an induced isomorphism
$\phi_*:\pi_1(X,x_0) \rightarrow \pi_1(X,\phi(x_0))$,
$[\alpha] \mapsto [\phi \circ \alpha]$.
Choose a path $\gamma:I \rightarrow X$ from $x_0$ to $\phi(x_0)$; this
induces an isomorphism
$\gamma_*:\pi_1(X,\phi(x_0)) \rightarrow \pi_1(X,x_0)$,
$[\alpha] \mapsto [\gamma\alpha\bar{\gamma}]$.
We now define a map
$\MCG_n(T^2) \rightarrow \Out(\pi_1(X,x_0))$ by
$[\phi] \mapsto [\gamma_* \phi_*]$.
One can show that this map is well-defined and is a homomorphism
(see \cite{Farb} Chapter 8.1).
In particular, the map is independent of the choice of the path
$\gamma$, since if we choose a different path $\gamma'$ then the
automorphisms $\gamma_*\phi_*$ and $\gamma_*' \phi_*$ of
$\pi_1(X,x_0)$ differ by the inner automorphism corresponding to
conjugation by $[\gamma'\bar{\gamma}] \in \pi_1(X,x_0)$.

\begin{remark}
A version of the Dehn-Nielsen-Baer theorem states
that the homomorphism
$\MCG_n(T^2) \rightarrow \Out(\pi_1(T^2 - \{p_1, \cdots, p_n\}))$ is
injective (see \cite{Farb} Theorem 8.8), and one can use this result
to obtain the expressions for the homomorphisms $p$ and $g$ in the
Birman sequence (\ref{seq:birman}).
\end{remark}

We define a right action of $\MCG_2(T^2)$ on the character variety
$R(T^2,2)$ by
\begin{align}
  \nonumber
  [\rho] \cdot f = [\rho \circ \tilde{f}],
\end{align}
where  $[\rho] \in R(T^2,2)$, $f \in \MCG_2(T^2)$, and
$\tilde{f} \in \Aut(\pi_1(T^2 - \{p_1, p_2\}))$ is a representative of
the image of $f$ under the homomorphism
$\MCG_2(T^2) \rightarrow \Out(\pi_1(T^2 - \{p_1, p_2\}))$.
We find that the action of $\MCG_2(T^2)$ on $R(T^2,2)$ is given by
\begin{align}
  \nonumber
  &[A,\, B,\, a,\, b] \cdot T_a =
  [A,\, BA,\, a,\, b], \\
  \nonumber
  &[A,\, B,\, a,\, b] \cdot T_b =
  [AB,\, B,\, a,\, b], \\
  \nonumber
  &[A,\, B,\, a,\, b] \cdot T_A =
  [A,\, aAB,\, a,\, Aaba^{-1}A^{-1}], \\
  \nonumber
  &[A,\, B,\, a,\, b] \cdot T_B =
  [a^{-1}BA,\, B,\, a,\, a^{-1}BbB^{-1}a], \\
  \nonumber
  &[A,\, B,\, a,\, b] \cdot \omega =
  [A^{-1},\, B^{-1},\, B^{-1}A^{-1}bAB,\, A^{-1}B^{-1}aBA].
\end{align}
The action of $\MCG_2(T^2)$ on $R(T^2,2)$ fixes the reducible locus
$\partial P_3$ of $R(T^2,2)$ as a set.
The homomorphism $p:B_2(T^2) \rightarrow \MCG_2(T^2)$ in the Birman
sequence (\ref{seq:birman}) induces a right action of
$B_2(T^2)$ on $R(T^2,2)$.

\section{Examples}
\label{sec:examples}

We will now compute generating sets for $I^\natural(Y,K)$ for several
example $(1,1)$-knots $K$ in lens spaces $Y$.
As described in the Introduction, we Heegaard-split $(Y,K)$ into
a pair of handlebodies $(U_1,A_1)$ and $(U_2,A_2)$.
The handlebodies are glued together via a homeomorphism
$\phi:(\partial U_1, \partial A_1) \rightarrow
(\partial U_2, \partial A_2)$,
which defines an element $f = [\phi]$ of the mapping class group
$\MCG_2(T^2)$.
We define a character variety $R(T^2,2)$ corresponding to the Heegaard
surface $(T^2,\{p_1,p_2\}) := (\partial U_1, \partial A_1)$,
and we define Lagrangians $L_1^\pi$ and  $L_2 = L_1 \cdot f$ in $R(T^2,2)$
corresponding to the handlebodies $(U_1,A_1)$ and $(U_2,A_2)$.
To obtain a generating set for $I^\natural(Y,K)$, we count the
intersection points $L_1^\pi \cap L_2$ and show that the intersection is
transverse at each point.
The calculations needed to accomplish this task rely on
the parameterizations
$L_1^\pi(\phi,\theta)$ and $L_1(\chi,\psi)$ of the Lagrangians $L_1^\pi$ and
$L_1$ given in Theorems
\ref{theorem:L1} and \ref{theorem:Ld}, together with the description
of the action of $\MCG_2(T^2)$ on $R(T^2,2)$ given in Section
\ref{ssec:action}.
To describe the intersection, we will use the coordinates
$(\hat{a},\hat{b})$ that we defined on the piece
$P_4 \subset R(T^2,2)$ in Section \ref{sssec:P4},
and the coordinates $(\alpha,\beta,\gamma)$ that we defined on the
piece $P_3 \subset R(T^2,2)$ in Section \ref{sssec:P3}.

\subsection{Trefoil in $S^3$}
\label{sec:trefoil-s3}

As shown in Figure \ref{fig:make-trefoil},
we can construct the trefoil in $S^3$ by gluing the two handlebodies
together using the mapping class group element
$f = s\beta_1\alpha_1^{-1} \in \MCG_2(T^2)$, where
$s := T_aT_b^{-1}T_a$ exchanges the longitude and meridian of $T^2$.
We first prove a Lemma that constrains the possible intersection
points of $L_1^\pi$ and $L_2 = L_1\cdot f$:

\begin{figure}
  \centering
  \includegraphics[scale=0.7]{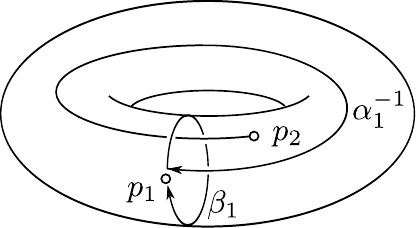}
  \caption{
    \label{fig:make-trefoil}
    The trefoil in $S^3$ is constructed by gluing together $(U_1,A_1)$
    and $(U_2,A_2)$ using the mapping class group element
    $f = s\beta_1\alpha_1^{-1}$.
  }
\end{figure}

\begin{lemma}
\label{lemma:trefoil}
If $L_1^\pi(\phi,\theta) = L_2(\chi,\psi)$, then $\chi = \pi/2$ and
either $\theta \in \{\pi/2,3\pi/2\}$ or $\phi \in \{0,\pi\}$
\end{lemma}

\begin{proof}
Define functions $h_1, h_2:R(T^2,2) \rightarrow \Reals$ by
\begin{align}
  \nonumber
  h_1([A,B,a,b]) &= \tr A, &
  h_2([A,B,a,b]) &= \tr Ba.
\end{align}
We evaluate the functions $h_1$ and $h_2$ at the points
$L_1^\pi(\phi,\theta)$ and $L_2(\chi,\psi)$.
If we require that each function give the same value at both points,
we obtain the desired result.
\end{proof}

\begin{theorem}
\label{theorem:example-trefoil}
The rank of $I^\natural(S^3,K)$ for the trefoil $K$ in $S^3$ is at most 3.
\end{theorem}

\begin{proof}
From Lemma \ref{lemma:trefoil}, we know that if
$L_1^\pi(\phi,\theta) = L_2(\chi,\psi)$ then $\chi = \pi/2$.
A calculation shows that
$L_2(\pi/2,\psi) = L_1(\pi/2,\psi) \cdot f = [A,B,a,b]$, where
\begin{align}
  \label{eqn:trefoil-1}
  A &= \qx, &
  B &= \sin 3\psi + \cos 3\psi\,\qz, \\
  \label{eqn:trefoil-2}
  a &=
  -\cos 2\psi\,\qx + \sin 2\psi\,\qy, &
  b &=
  -\cos 4\psi\,\qx - \sin 4\psi\,\qy.
\end{align}

We will first show that the intersection $L_1^\pi \cap L_2$ takes place
entirely in the piece $P_4$.
Suppose $L_2(\pi/2,\psi)$ lies in the piece $P_3$.
Then the matrices $A$ and $B$ in equation (\ref{eqn:trefoil-1}) must
commute, so $\cos 3\psi = 0$, corresponding to
$\psi \in \{\pm \pi/6, \pm \pi/2\}$.
From equations (\ref{eqn:trefoil-1}) and (\ref{eqn:trefoil-2}), we
find that
\begin{align}
  \nonumber
  \gamma(L_2(\pi/2,\pm \pi/6)) &= -\pi/6, &
  \gamma(L_2(\pi/2,\pm\pi/2)) &= \pi/2.
\end{align}
But Theorem \ref{theorem:L1} states that all of
the points in $L_1^\pi \cap P_3$ have $\gamma=0$.
It follows that $L_1^\pi$ does not intersect $L_2$ in the piece $P_3$.

We now consider the intersection $L_1^\pi \cap L_2$ in the piece $P_4$.
Using equations (\ref{eqn:trefoil-1}) and (\ref{eqn:trefoil-2}), we
find that the $(\hat{a},\hat{b})$ coordinates of $L_2(\pi/2,\psi)$ are
\begin{align}
  \label{eqn:trefoil-d1}
  \hat{a}(L_2(\pi/2,\psi)) &=
  (-\cos 2\psi,\, \sin 2\psi,\, 0), &
  \hat{b}(L_2(\pi/2,\psi)) &=
  (-\cos 4\psi,\, -\sin 4\psi,\, 0)
\end{align}
for $\psi \in (-\pi/6,\pi/6)$, and
\begin{align}
  \label{eqn:trefoil-d2}
  \hat{a}(L_2(\pi/2,\psi)) &=
  (-\cos 2\psi,\, -\sin 2\psi,\, 0), &
  \hat{b}(L_2(\pi/2,\psi)) &=
  (-\cos 4\psi,\, \sin 4\psi,\, 0)
\end{align}
for $\psi \in (-\pi/2,-\pi/6) \cup (\pi/6,\pi/2)$.
From Lemma \ref{lemma:trefoil}, we know that either
$\theta \in \{\pi/2,3\pi/2\}$ or $\phi \in \{0,\pi\}$.
But $\phi=0$ and $\phi=\pi$ correspond to the double-point of $L_2$,
which lies in $P_3$, and we have already shown that $L_1^\pi$ does not
intersect $L_2$ in $P_3$.
Thus $\theta \in \{\pi/2,3\pi/2\}$.
Substituting $\theta=\pi/2$ and $\theta=3\pi/2$ into
the expressions for the $(\hat{a},\hat{b})$ coordinates of
$L_1^\pi(\phi,\theta)$ given in Theorem \ref{theorem:L1}, we find that
\begin{align}
  \label{eqn:trefoil-s1}
  &\hat{a}(L_1^\pi(\phi,\pi/2)) =
  (-\sin(\phi+\nu),\, -\cos(\phi+\nu),\, 0), &
  &\hat{b}(L_1^\pi(\phi,\pi/2)) =
  (\sin(\phi-\nu),\, \cos(\phi-\nu),\, 0), \\
  \label{eqn:trefoil-s2}
  &\hat{a}(L_1^\pi(\phi,3\pi/2)) =
  (\sin(\phi+\nu),\, \cos(\phi+\nu),\, 0), &
  &\hat{b}(L_1^\pi(\phi,3\pi/2)) =
  (-\sin(\phi-\nu),\, -\cos(\phi-\nu),\, 0).
\end{align}
From equations (\ref{eqn:trefoil-d1})--(\ref{eqn:trefoil-s2}),
it follows that the intersection
$L_1^\pi \cap L_2$ fact takes place in a torus
$T^2 - \bar{\Delta} \subset S^2 \times S^2 - \bar{\Delta}$, where
$\bar{\Delta} \subset T^2$ is the antidiagonal.
In Figure \ref{fig:trefoil-intersection} we use
equations (\ref{eqn:trefoil-d1})--(\ref{eqn:trefoil-s2}) to plot the
intersection of $L_1^\pi$ and $L_2$ in $T^2 - \bar{\Delta}$.
We see that $L_1^\pi$ and $L_2$ intersect in three points.

We will now show that the intersection is transverse at each
of these three points.
A calculation shows that at each point we have
\begin{align}
  \nonumber
  \partial_\phi h_1(L_1^\pi(\phi,\theta)) &= 0, &
  \partial_\theta h_1(L_1^\pi(\phi,\theta)) &\neq 0, &
  \partial_\chi h_1(L_2(\chi,\psi)) &= 0, &
  \partial_\psi h_1(L_2(\chi,\psi)) &= 0, \\
  \nonumber
  \partial_\phi h_2(L_1^\pi(\phi,\theta)) &= 0, &
  \partial_\theta h_2(L_1^\pi(\phi,\theta)) &= 0, &
  \partial_\chi h_2(L_2(\chi,\psi)) &\neq 0, &
  \partial_\psi h_2(L_2(\chi,\psi)) &= 0.
\end{align}
These equations, together with Figure \ref{fig:trefoil-intersection},
show that the intersection is transverse at each intersection point.
\end{proof}

For knots $K$ in $S^3$, one can show
(see \cite{Hedden-1}, Section 12.1) that
\begin{align}
  \nonumber
  \rank I^\natural(S^3,K) \geq \sum_i |a_i|,
\end{align}
where $a_i$ is the coefficient of $t^i$ in the Alexander polynomial
$\Delta_K(t)$ of $K$:
\begin{align}
  \nonumber
  \Delta_K(t) = \sum_i a_i t^i.
\end{align}
This inequality, together with Theorem
\ref{theorem:example-trefoil}, gives the singular instanton homology
for the trefoil.
This result was already known, since, as shown by Kronheimer and
Mrowka, the singular instanton homology of an alternating knot in
$S^3$ is isomorphic to the reduced Khovanov homology of its mirror
\cite{Kronheimer-3}.

\begin{figure}
  \centering
  \includegraphics[scale=0.7]{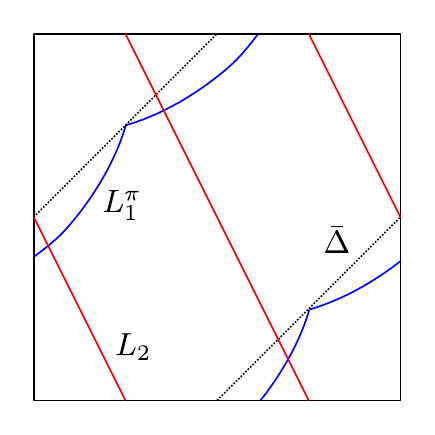}
  \caption{
    \label{fig:trefoil-intersection}
    The trefoil in $S^3$.
    The space depicted is
    $T^2 - \bar{\Delta} \subset S^2 \times S^2 - \bar{\Delta}$.
    Shown are the Lagrangian $L_1^\pi$, the Lagrangian $L_2$,
    and the antidiagonal $\bar{\Delta}$.
  }
\end{figure}

\subsection{Unknot in $L(p,1)$ for $p$ not a multiple of 4}
\label{sec:unknot}

We can construct the unknot $U$ in the lens space $L(p,1)$ by gluing
the two handlebodies together using the mapping class group element
$f = T_a^p \in \MCG_2(T^2)$.
The following is a restatement of Theorem \ref{theorem:intro-unknot}
from the Introduction:

\begin{theorem}
If $p$ is not a multiple of 4, then the rank of $I^\natural(L(p,1),U)$
for the unknot $U$ in the lens space $L(p,1)$ is at most $p$.
\end{theorem}

\begin{proof}
A calculation shows that
$L_2(\chi,\psi) = L_1(\chi,\psi) \cdot f = [A,B,a,b]$,
where
\begin{align}
  \label{eqn:unknot}
  A &= \cos \chi + \sin\chi\,\qz, &
  B &= \cos p\chi + \sin p\chi\,\qz, &
  a &= b^{-1} = \cos\psi\,\qx + \sin\psi\,\qz.
\end{align}
Since $A$ and $B$ commute, the Lagrangian $L_2$ lies
in the piece $P_3$.
From equation (\ref{eqn:unknot}), it follows that the
$(\alpha,\beta,\gamma)$ coordinates of the point $L_2(\chi,\psi)$ are
\begin{align}
  \nonumber
  \alpha(L_2(\chi,\psi)) &= \chi, &
  \beta(L_2(\chi,\psi)) &= p\chi, &
  \gamma(L_2(\chi,\psi)) &= \psi.
\end{align}
Comparing with the parameterization of $L_1^\pi$ in $P_3$ given in
Theorem \ref{theorem:L1}, we find that
the intersection $L_1^\pi \cap L_2$ in fact takes place in the
pillowcase $P_3 \cap \{\gamma=0\}$.
In Figure \ref{fig:unknot-lens-space} we plot the intersection of
$L_1^\pi$ with $L_2$ in the pillowcase $P_3 \cap \{\gamma=0\}$
for $p=1,2,3$.
We find that if $p$ is not a multiple of 4 then we obtain a generating
set with $p$ generators.
If $p$ is a multiple of $4$ then $L_1^\pi \cap L_2$ contains the
double-point $(\alpha,\beta,\gamma) = (\pi/2,0,0)$ of $L_1^\pi$, and thus
our scheme for counting generators fails.

We will now show that the intersection is transverse at each
intersection point.
Define functions
\begin{align}
  \nonumber
  h_1([A,B,a,b]) &= \tr Aa, &
  h_2([A,B,a,b]) &= \tr Ba.
\end{align}
A straightforward calculation shows that at each point of
$L_1^\pi \cap L_2$ we have that
\begin{align}
  \nonumber
  \partial_\phi h_1(L_1^\pi(\phi,\theta)) &= 0, &
  \partial_\theta h_1(L_1^\pi(\phi,\theta)) &\neq 0, &
  \partial_\chi h_1(L_2(\chi,\psi)) &= 0, &
  \partial_\psi h_1(L_2(\chi,\psi)) &\neq 0, \\
  \nonumber
  \partial_\phi h_2(L_1^\pi(\phi,\theta)) &= 0, &
  \partial_\theta h_2(L_1^\pi(\phi,\theta)) &= 0, &
  \partial_\chi h_2(L_2(\chi,\psi)) &= 0, &
  \partial_\psi h_2(L_2(\chi,\psi)) &\neq 0.
\end{align}
These equations,
together with Figure \ref{fig:unknot-lens-space}, show that the
intersection is transverse at each point of $L_1^\pi \cap L_2$.
\end{proof}

For the case $p=1$ we have that $L(p,1) = S^3$, and our results imply
that the unknot in $S^3$ has a generating set with a single generator.
Since there is a single generator, there are no differentials, and
this result amounts to a calculation of the singular instanton
homology.

\begin{remark}
It is interesting to note that for the unknot $U$ in the lens space
$Y = L(p,q)$, the knot Floer homology $\widehat{HFK}(Y,U)$ has rank
$p$ (see \cite{Hedden-0}).
\end{remark}

\begin{figure}
  \centering
  \includegraphics[scale=0.7]{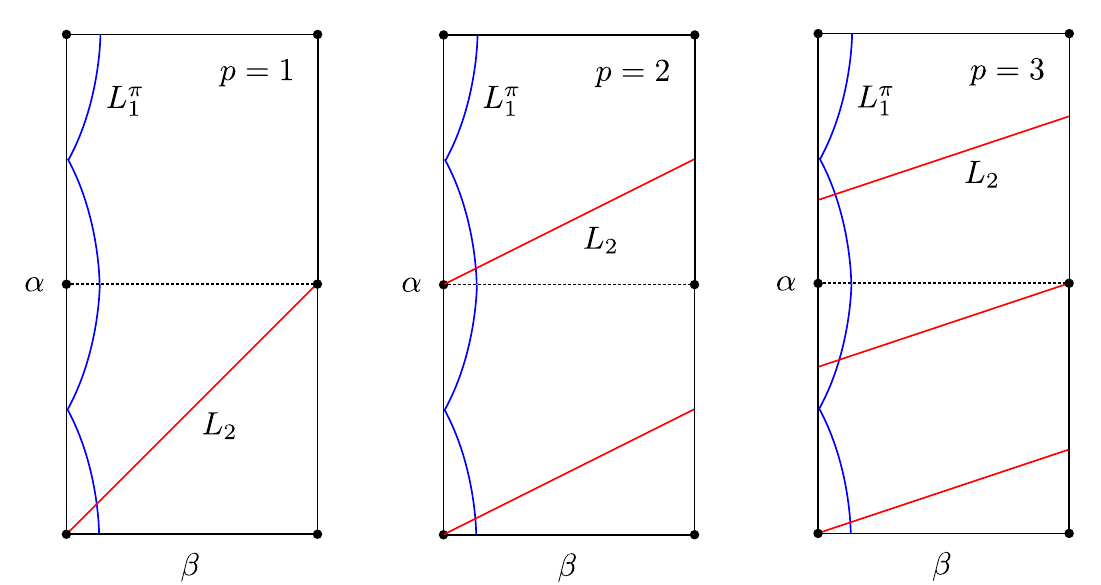}
  \caption{
    \label{fig:unknot-lens-space}
    The unknot in $L(p,1)$ for
    $p=1,2,3$.
    The space depicted is the pillowcase
    $P_3 \cap \{\gamma = 0\}$.
    Shown are the Lagrangians $L_1^\pi$ and $L_2$.
  }
\end{figure}

\subsection{Simple knot in $L(p,1)$ in homology class
$1 \in \Ints_p = H_1(L(p,1);\Ints)$}

\begin{definition}
A knot $K$ in a lens space $L(p,q)$ is said to be \emph{simple} if the
lens space has a Heegaard splitting into solid tori $U_1$ and $U_2$
with meridian disks $D_1$ and $D_2$ such that $D_1$ intersects $D_2$
in $p$ points and $K \cap U_i$ is an unknotted arc in disk $D_i$ for
$i=1,2$ (see \cite{Hedden-0}).
\end{definition}

One can show that there is exactly one simple knot in each nonzero
homology class of $H_1(L(p,q);\Ints) = \Ints_p$ \cite{Hedden-0}.
For the case $q=1$, we can view the lens space $L(p,1)$ is a circle
bundle over $S^2$, and a loop that winds $n$ times around a circle
fiber is a simple knot in homology class
$n \in \Ints_p = H_1(L(p,1);\Ints)$.
For $p \geq 2$, we can construct the simple knot $K$ in the lens space
$L(p,1)$ corresponding to the homology class
$1 \in \Ints_p = H_1(L(p,1);\Ints)$ by gluing the two handlebodies
together using the mapping class group element
$f = \alpha_1^{-1} T_a^p \in \MCG_2(T^2)$.
We first prove a result that constrains the possible intersection
points of $L_1^\pi$ and $L_2 = L_1 \cdot f$:

\begin{lemma}
\label{lemma:lens-simple}
If $L_1^\pi(\phi,\theta) = L_2(\chi,\psi)$ then $\phi = \pi/2$ and
$(\chi,\psi) \in \{(\chi_0,\psi_0), \cdots,
(\chi_{p-1},\psi_{p-1})\}$, where
$\chi_n := (n+1/2)(\pi/p)$ and
$\psi_n := (-1)^{n+1}(\pi/2 - \epsilon)$.
\end{lemma}

\begin{proof}
Define a function
$h_1:R(T^2,2) \cap \{\tr Ab \neq 0\} \rightarrow \Reals$ and functions
$h_2,h_3:R(T^2,2) \rightarrow \Reals$ by
\begin{align}
  \nonumber
  h_1([A,B,a,b]) &= -\frac{\tr Aa}{\tr Ab}, &
  h_2([A,B,a,b]) &= \tr Ba, &
  h_3([A,B,a,b]) &= \tr B.
\end{align}
Using straightforward calculations, one can show that if
$h_3(L_1^\pi(\phi,\theta)) = h_3(L_2(\chi,\psi))$ then
$(\tr Ab)(L_2(\chi,\psi)) \neq 0$, and thus the function $h_1$ is
defined everywhere on $L_1^\pi \cap L_2$.
We evaluate the functions $h_1$, $h_2$, and $h_3$ 
at the points $L_1^\pi(\phi,\theta)$ and $L_2(\chi,\phi)$.
If we require that each function give the same value at both points,
we obtain the desired result.
\end{proof}

The following is a restatement of Theorem \ref{theorem:intro-simple}
from the Introduction:

\begin{theorem}
If $K$ is the unique simple knot in the lens space $L(p,1)$
representing the homology class $1 \in \Ints_p = H_1(L(p,1);\Ints)$,
then the rank of $I^\natural(L(p,1),K)$ is at most $p$.
\end{theorem}

\begin{proof}
We will argue that each of the $p$ potential intersection points
described by Lemma \ref{lemma:lens-simple} is an actual intersection
point.
A calculation shows that
$L_2(\chi_n,\psi_n) = L_1(\chi_n,\psi_n) \cdot f = [A,B,a,b]$, where
\begin{align}
  \label{eqn:lens-simple-1}
  A &= \cos\chi_n + \sin\chi_n\,\qx, &
  B &= \cos\epsilon + \sin\epsilon\,\qz, \\
  \label{eqn:lens-simple-2}
  a &= (-1)^{n+1}(\cos\epsilon\,\qx + \sin\epsilon\,\qy), &
  b &=
  (-1)^n\cos\epsilon\,\qx +
  \sin\epsilon \cos \eta_n\,\qy +
  \sin\epsilon \sin \eta_n\,\qz,
\end{align}
and $\eta_n := (1 + n(p+2))(\pi/p)$.
We note that $A$ and $B$ do not commute,
since the coefficient of $\qx$ in $A$ and the coefficient of
$\qz$ in $B$ are both nonzero, so the intersection
$L_1^\pi \cap L_2$ takes place entirely in the piece $P_4$.
From equations (\ref{eqn:lens-simple-1}) and
(\ref{eqn:lens-simple-2}), we find that the
$(\hat{a},\hat{b})$ coordinates of $L_2(\chi_n,\psi_n)$ are given by
\begin{align}
  \label{eqn:lens-simple-d}
  \hat{a}(L_2(\chi_n,\psi_n)) &=
  (-1)^{n+1}(\cos\epsilon,\, \sin\epsilon,\,0), &
  \hat{b}(L_2(\chi_n,\psi_n)) &=
  ((-1)^n\cos\epsilon,\,
  \sin\epsilon \cos \eta_n,\,
  \sin\epsilon \sin \eta_n).
\end{align}
From Lemma  \ref{lemma:lens-simple}, we know that if
$L_1^\pi(\phi,\theta) = L_2(\chi,\psi)$ then $\phi = \pi/2$.
Substituting $\phi = \pi/2$ into the expressions for the
$(\hat{a},\hat{b})$ coordinates of $L_1^\pi(\phi,\theta)$ given in
Theorem \ref{theorem:L1}, we find that
\begin{align}
  \label{eqn:lens-simple-s1}
  \hat{a}(L_1^\pi(\pi/2,\theta)) &=
  (\cos\epsilon,\,\sin\epsilon,\,0), &
  \hat{b}(L_1^\pi(\pi/2,\theta)) &=
  (-\cos\epsilon,\,-\sin\epsilon \cos\bar{\theta},\,
  \sin\epsilon\sin\bar{\theta})
\end{align}
for $\theta \in (0,\pi)$, and
\begin{align}
  \label{eqn:lens-simple-s2}
  \hat{a}(L_1^\pi(\pi/2,\theta)) &=
  (-\cos\epsilon,\,-\sin\epsilon,\,0), &
  \hat{b}(L_1^\pi(\pi/2,\theta)) &=
  (\cos\epsilon,\,\sin\epsilon \cos\bar{\theta},\,
  \sin\epsilon\sin\bar{\theta})
\end{align}
for $\theta \in (\pi,2\pi)$, where $\bar{\theta}$ is defined such that
\begin{align}
  \nonumber
  \cos\bar{\theta} &=
  \frac{\cos^2\epsilon\cos^2\theta - \sin^2\theta}
       {\cos^2\epsilon\cos^2\theta + \sin^2\theta}, &
  \sin\bar{\theta} &=
  \frac{\cos\epsilon \sin 2\theta}
       {\cos^2\epsilon\cos^2\theta + \sin^2\theta}.
\end{align}
It is straightforward to verify that for small enough values of
$\epsilon$, the maps $(0,\pi) \rightarrow (0,2\pi)$,
$\theta \mapsto \bar{\theta}$ and
$(\pi,2\pi) \rightarrow (0,2\pi)$,
$\theta \mapsto \bar{\theta}$ are diffeomorphisms.
Thus we can always solve equations
(\ref{eqn:lens-simple-d})--(\ref{eqn:lens-simple-s2}) to obtain a
unique value of $\theta$ such that
$L_1^\pi(\pi/2,\theta) = L_2(\chi_n,\psi_n)$.
Specifically, if $n$ is even, then $\theta \in (0,\pi)$ is given by
$\bar{\theta}(\theta) = \eta_n$, and
if $n$ is odd then $\theta \in (\pi,2\pi)$ is given by
$\bar{\theta}(\theta) = \pi - \eta_n$.
We conclude that $L_1^\pi$ and $L_2$ intersect in $p$ points.

We will now show that $L_1^\pi$ intersects $L_2$ transversely at
each of these $p$ points.
A straightforward calculation shows that at each point of
$L_1^\pi \cap L_2$ we have
\begin{align}
  \nonumber
  \partial_\phi h_1(L_1^\pi(\phi,\theta)) &\neq 0, &
  \partial_\theta h_1(L_1^\pi(\phi,\theta)) &= 0, &
  \partial_\chi h_1(L_2(\chi,\psi)) &= 0, &
  \partial_\psi h_1(L_2(\chi,\psi)) &= 0, \\
  \nonumber
  \partial_\phi h_2(L_1^\pi(\phi,\theta)) &= 0, &
  \partial_\theta h_2(L_1^\pi(\phi,\theta)) &= 0, &
  \partial_\chi h_2(L_2(\chi,\psi)) &\neq 0, &
  \partial_\psi h_2(L_2(\chi,\psi)) &= 0, \\
  \nonumber
  \partial_\phi h_3(L_1^\pi(\phi,\theta)) &= 0, &
  \partial_\theta h_3(L_1^\pi(\phi,\theta)) &= 0, &
  \partial_\chi h_3(L_2(\chi,\psi)) &= 0, &
  \partial_\psi h_3(L_2(\chi,\psi)) &\neq 0.
\end{align}
These equations, together with Theorem \ref{theorem:L1}, show that the
intersection is transverse at each point.
\end{proof}

For the case $p=0$, the knot we have constructed is
$K = S^1 \times \{pt\}$ in $S^1 \times S^2$, and our above result
implies that this knot has a generating set with zero generators.
This result holds even in the absence of the perturbation, since there
are no homomorphisms
$\rho:\pi_1(S^1 \times S^2 - K) \rightarrow SU(2)$ that take loops
around $K$ to traceless matrices.

For the case $p=1$, the knot we have constructed is the unknot in
$S^3$, and we have have reproduced the result of Section
\ref{sec:unknot} for this knot.

\begin{remark}
It is interesting to note that for a simple knot $K$ in the lens space
$Y = L(p,q)$, the knot Floer homology $\widehat{HFK}(Y,K)$ has rank
$p$ (see \cite{Hedden-0}).
\end{remark}

\section*{Acknowledgments}

The author would like to express his gratitude towards Ciprian
Manolescu for providing invaluable guidance.
The author was partially supported by NSF grant number
DMS-1708320.

\end{document}